\documentclass[11pt,reqno]{amsart}
\usepackage{amsmath}
\usepackage{amssymb,amsmath}
\usepackage{color}
\usepackage{tikz}
\usepackage{graphicx}

\let\pa=\partial
\let\al=\alpha

\let\g=\gamma

\let\e=\varepsilon

\let\f=\frac
\let\p=\psi

\let\D=\Delta

\let\vth=\vartheta

\def\cT{{\mathcal T}}

\def\bbT{\mathbb{T}}
\def\R{\mathbb{R}}

\def\Re{\mathrm{Re}}

\def\na{\nabla}

\def\p{\partial}

\def\non{\nonumber}

\def\C{\mathop{\bf C\kern 0pt}\nolimits}
\def\DD{\mathop{\bf D\kern 0pt}\nolimits}
\def\K{\mathop{\bf K\kern 0pt}\nolimits}
\def\N{\mathop{\bf N\kern 0pt}\nolimits}
\def\Q{\mathop{\bf Q\kern 0pt}\nolimits}

\def\div{\mathop{\rm div}\nolimits}

\newcommand{\lan}{\langle}
\newcommand{\ran}{\rangle}

\newcommand{\Z}{{\mathbf Z}}

\newcommand{\beq}{\begin{equation}}
\newcommand{\eeq}{\end{equation}}
\newcommand{\ben}{\begin{eqnarray}}
\newcommand{\een}{\end{eqnarray}}
\newcommand{\beno}{\begin{eqnarray*}}
\newcommand{\eeno}{\end{eqnarray*}}

\newtheorem{rmk}{Remark}[section]

\newtheorem{theorem}{Theorem}[section]

\newtheorem{lemma}[theorem]{Lemma}
\newtheorem{proposition}[theorem]{Proposition}

\allowdisplaybreaks

\begin{document}

\title[Transition threshold  for  2D Boussinesq equations]
{ Transition threshold  of   Couette flow for  2D Boussinesq equations}

\author{Xiaoxia Ren}
\address{Department of Mathematics and Physics, North China Electric Power University, 102206, Beijing, P. R. China}
\email{xiaoxiaren@ncepu.edu.cn}

\author{Dongyi Wei}
\address{School of Mathematical Science, Peking University, 100871, Beijing, P. R. China}
\email{jnwdyi@pku.edu.cn}

\maketitle
\begin{abstract}
In this paper,  we prove   the   stability threshold  of $\alpha\leq \f13$ for 2D Boussinesq equations around the Couette flow   in $\mathbb{T}\times \mathbb{R}$ with Richardson number $\g^2>\f14$ and different   viscosity $\nu$ and thermal diffusivity $\mu$. More precisely, if $\|v_{in}-(y,0)\|_{H^{s+1/2}}+
\|\rho_{in}+\g^2 y-1\|_{H^{s+1/2}}\leq c(\min\{\nu,\mu\})^{1/3}$, $\f{\nu+\mu}{2\g\sqrt{\nu \mu} }< 2-\e$, $s>3/2$, then the asymptotic
 stability holds.   This stability threshold is consistent with the optimal stability  threshold for the 2D  Navier-Stokes  equations in Sobolev space. And in the sense of inviscid damping effect,  the regularity assumption of the initial data should be sharp.

\end{abstract}
%

\section{Introduction}
In this paper, we consider the 2D incompressible Boussinesq equations  in $\mathbb{T}\times \mathbb{R}$
\ben\label{NS-Boussinesq} 
\left\{
\begin{array}{l}
\p_t v-\nu\Delta v+\na \pi=-v\cdot\na v-(0,\rho g),  \\
\p_t\rho-\mu\Delta \rho=-v\cdot\na \rho, \\
\div v=0, \\
v(x, y,0)=v_0(x, y),\ \ \rho(x, y, 0)=\rho_0(x, y),
\end{array}\right.
\een
here
 $v=(v_1, v_2)$ is the velocity field,  $p$ is the pressure, $\rho$ is the temperature (density), $g=1$ is the gravitational constant, 
 $\nu$ is the viscosity coefficient and $\mu$ is the thermal diffusivity. The first two equations in \eqref{NS-Boussinesq}
are the incompressible Navier-Stokes equation with buoyancy forcing in the vertical
direction. The third equation is a balance of the temperature convection and diffusion.
 The 2D Boussinesq system  is related to the study of fluid motion such as ocean circulation and atmospheric fronts, as well as other astrophysical movements \cite{GAE,PED,Ma}. In addition, the system is also known for its close connection with basic models of 3D incompressible flows, such as Euler and Navier-Stokes equations \cite{LPZ}.

A steady state of \eqref{NS-Boussinesq}  is Couette flow with a linear function of the vertical component
\beq\label{steady state}
v_s=(y,0),\quad \rho_s=-\g^2 y+1,\quad p_s=\f12\g^2y^2-y+C_0,
\eeq
where the constant $\g^2$ is the Richardson number.
The Richardson number stands as a pivotal control parameter in determining the stability of stratified shear flows. According to the Miles-Howard theorem in \cite{Howard, Miles}, any flow in the inviscid, non-diffusive limit is assured of linear stability provided that the local Richardson number surpasses the threshold value of  $\f14$ everywhere. Conversely, when the Richardson number dips below this critical value of  $\f14$, unstable modes may emerge, as elaborated in \cite{Drazin}.

In this paper, we investigate the asymptotic stability of the Couette flow in the scenario where the Richardson number $\g^2> \f14$. Let $v=\tilde{u}+v_s$, $\pi=\tilde{p}+p_s$ and $\rho=\g^2\tilde{\vth}+\rho_s$, then the perturbation $(\tilde{u},\tilde{p},\tilde{\vth})$ satisfies
\begin{equation}\label{eq: u sys}
 \left\{\begin{array}{l}
\pa_t\tilde{u}+y\pa_x\tilde{u}+(\tilde{u}_2,0)+\tilde{u}\cdot\na \tilde{u}-\nu\D \tilde{u}+\na \tilde{p}=- (0,
    \g^2\tilde{\vth}),\\
\pa_t\tilde{\vth}+y\pa_x\tilde{\vth}+\tilde{u}\cdot\na\tilde{\vth}-\mu\D \tilde{\vth}-\tilde{u}_2=0,\\
\div \tilde{ u}=0, \\
\tilde{u}|_{t=0}=\tilde{u}_{in}(x,y),\ \tilde{\vth}|_{t=0}=\tilde{\vth}_{in}(x,y).
\end{array}\right.
\end{equation}
where the pressure $ \tilde{p}$ is determined by
$$\tilde{p}=(-\Delta)^{-1} [2\p_1 \tilde{u}_2+\g^2 \p_2 \tilde{\vth}+\p_i \p_j (\tilde{u}_i \tilde{u}_j )].$$

The transition from laminar flow to a turbulent state has been a classical problem in fluid dynamics since the early experiments conducted by Reynolds \cite{Re}, {which is mainly concerned with how the laminar flows become unstable and transit to turbulence at high Reynolds number.  
Some laminar flows are
 linearly stable for any Reynolds number \cite{DR,Rom}, but could be unstable and transition to turbulence for small perturbations at high Reynolds number \cite{DHB,TA}.
 It was observed by Kelvin in \cite{Kel} that the basin of
 attraction of the laminar flow shrinks as $Re\to\infty$ so that the flow could become nonlinearly unstable for small but finite
 perturbations.
 Thus, an important question firstly proposed by Trefethen et al. \cite{T} is to study the transition threshold: \textit{how much disturbance will lead to the instability of the flow and the dependence of disturbance on the Reynolds number}.}
To understand the transition mechanism, we are concerned with 
the following mathematical version of stability threshold problem
formulated by Bedrossian, Germain and Masmoudi \cite{BGM-BAMS}:

Given a norm $\|\cdot\|_{X}$, determine a $\al=\al(X)$, such that  (here $\kappa=\min\{\nu,\mu\}$)
\begin{align*}
\|(\tilde{u}_0,\tilde{\vth}_0)\|_{X}\leq \kappa^{\al}\Rightarrow \text{stability},\nonumber\\
\|(\tilde{u}_0,\tilde{\vth}_0)\|_{X}\gg \kappa^{\al}\Rightarrow \text{instability}.
\end{align*}
The exponent $\al $ is referred to as the transition threshold.


When the temperature $ \tilde{\vth}=0$, the Boussinesq equations \eqref{eq: u sys} simplify to the Navier-Stokes equations. 
Significant efforts have been made to quantify the stability threshold for Couette flow in Navier-Stokes equations. 
In the 2D domain $\mathbb{T}\times\mathbb{R}$, Bedrossian,  Vicol and Wang \cite{B} established a stability threshold of $\al\leq\f12$. This threshold has subsequently been refined to $\al\leq\f13$ in \cite{M, WZ23}.
In a finite channel $\mathbb{T}\times[-1,1]$, Chen, Li, Wei and Zhang \cite{CL} demonstrated a stability threshold of  $\al\leq\f12$
for the 2D Navier-Stokes equations with non-slip boundary condition. Additionally, Wei and Zhang \cite{WZ2} proved a threshold of $\al\leq\f13$
for the same equations with a Navier-slip boundary condition.
 The 3D transition threshold problem for the Navier-Stokes equations is notably distinct and more challenging due to the lift-up effect. For further insights, see, for example, \cite{BGM, BGM2, CW, WZ}.

 The Boussinesq system is different from the Navier-Stokes equations.
On the one hand, the buoyancy force in the velocity equation could drive the growth of the energy
and more generally the growth of the Sobolev norms;
 On the other hand, the temperature field has a stabilizing effect \cite{Zhai, DWZ, ZZi, NZ, MZZ}. For Euler-Boussinesq system (i.e. $\nu=\mu=0$), Lin and Yang \cite{YangLin} obtain the  linear stability for steady state \eqref{steady state} with $\g^2>\f14$. 
One can also refer to \cite{Bianchini} for linear case and \cite{BBZ} for nonlinear case. 
For NS-Boussinesq system with $\mu=0$, Masmoudi, Said-Houari, and Zhao \cite{MSZ} proved  the asymptotically stability for Couette flow in Gevrey class.
When $\g^2>\f14$ with $\nu\sim\mu$, the stability  threshold $\al=\f12$ is obtained in    \cite{Zhai, RW}; when $\g^2=0$, one can refer to \cite{DWZ, ZZi, NZ} for stability in  $\bbT\times \mathbb{R}$  and \cite{MZZ} for stability in $\mathbb{T}\times[-1,1] $. 
 For asymptotic stability when $\g$ is small
 one may refer to \cite{Zillinger1}. 
 For the 2D case when the domain is replaced by  $\mathbb{R}^2$, one may refer to \cite{A}. For the 3D case, one may refer to \cite{ZN} for linear stability and \cite{ZZW} for  nonlinear stability. 
 

To state the main result, we change the coordinates to  $X=x-yt, \   Y=y,$ according to
 the characteristics of the Couette flow $(y,0)$,
then
$$\ \  \p_x=\p_X, \ \ \p_y=\p_Y-t\p_X.
$$
Set 
\begin{align*}
u(t,X,Y)=\tilde{u}(t,x,y), \ \ \ \vth(t,X,Y)=\tilde{\vth}(t,x,y), \ \ \ p(t,X,Y)=\tilde{p}(t,x,y) \end{align*}
and \eqref{eq: u sys} becomes
\begin{equation}\label{eq: u new}
 \left\{\begin{array}{l}
u_t-\nu\D_Lu+(u_2,0)+\na_L p_L+ (
   0,\g^2\vth)=F,\\
\vth_t-\mu\D_L \vth-u_2=G,\\
\div_L u=0,\\
u|_{t=0}=u_{in}(x,y),\ \vth|_{t=0}=\vth_{in}(x,y).
\end{array}\right.
\end{equation}
where
\begin{align*}
\Delta_L&=\pa_X^2+(\pa_Y-t\pa_X)^2, \quad \na_L=(\na_L^1, \na_L^2)=(\pa_X,\pa_Y-t\pa_X).\quad \div_L u=\na_L^1 u_1+\na_L^2u_2,\\
p_L&=(-\Delta_L)^{-1}(2\pa_Xu_2+\g^2 (\pa_Y-t\pa_X) \vth), \quad F=-u\cdot\na_L u-\na p_{NL}, \quad G=-u\cdot\na_L \vth,
 \end{align*}
with 
\begin{align*}
-\Delta_Lp_{NL}&=\div_L(u\cdot\na_Lu)=\na_L^iu_j \na_L^ju_i=(\na_L^1u_1)^2+(\na_L^2u_2)^2+2\na_L^1u_2\na_L^2u_1\\&=
2((\na_L^1u_1)^2+\na_L^1u_2\na_L^2u_1)=
2((\na_L^2u_2)^2+\na_L^1u_2\na_L^2u_1)=2\div_L(u_2\na_L^2u).
 \end{align*}

Our stability result is stated as follows.
\begin{theorem}
\label{th:main}
Let $\kappa=\min\{\nu,\mu\}\in(0, 1)$, $\g>\f12$,  $s>3/2$ and  $\f{\nu+\mu}{2\g\sqrt{\nu \mu} }< 2-\e$, $0<\e<1/\g$. There exist  constants $c>0$, $\epsilon>0$  (depending only on $\g, s,\e $) such that if
$
\| (u_{in}, \vth_{in})\|_{H^{s+1/2}}\leq c\kappa^\f13,
$
then the solution  to system \eqref{eq: u new} is global in time and satisfies the following:

(1) global stability estimate in $  H^{s+1/2}$:
 \begin{align}\label{stability}
\| \lan\na_L\ran^\f12 (u, \vth)\|_{H^s }\le  C\| (u_{in}, \vth_{in})\|_{H^{s+1/2}},
\end{align}

(2) inviscid damping estimate:
\begin{align}\label{2}
    \|(u_1)_{\neq}\|_{L^2}+\lan t\ran\|u_2\|_{L^2}+\|\vth_{\neq}\|_{L^2}\leq C\lan t\ran^{-\f12}  e^{-\epsilon\kappa^\f13 t}\| (u_{in}, \vth_{in})\|_{H^{s+1/2}},
    \end{align}
    
  (3) enhanced dissipation estimate:
  \begin{align}\label{3}
     \| \lan\na_L\ran^\f12 (u_{\neq}, \vth_{\neq})\|_{H^s }\leq Ce^{-\epsilon\kappa^\f13 t}\| (u_{in}, \vth_{in})\|_{H^{s+1/2}},
    \end{align}
where $f_{\neq}=P_{\neq}f=f-P_0f$ and $P_0f=\f{1}{2\pi}\int_{\mathbb{T}}f(x,y)dx$ denote the non-zero mode and
 the zero mode, $ C$ is a constant depending only on $\g, s,\e $.

\end{theorem}

\begin{rmk} Even for the linearized 2D Euler-Boussinesq system, the initial velocity is required
to be in $H^2$ in order to obtain the inviscid damping estimate of the velocity, especially the $\lan t\ran^{-\f32} $ decay of $u_2$ as in Theorem \ref{th:main}, see \cite{YangLin}. 
In this sense, the regularity assumption of the initial data should be sharp. 
 
\end{rmk}

\begin{rmk}   Very recently, Knobel  \cite{Knobel}  also improved the stability threshold for the Boussinesq equations to  $1/3$ for $\g>1/2$
in $H^N (N\geq 13)$ when  $\nu=\mu$ using symmetric variables.
However, there are several key distinction   between \cite{Knobel}  and our paper. Indeed,  we do not  employ  symmetric variables,  we use the energy estimate as in \cite{RW} to removed the technical restriction $\nu=\mu$. Moreover, by more delicate  multiplier estimate, we  only required the  initial data to be in  $H^{s+\f12} (s>\f32)$.
\end{rmk}

In \cite{WZ23}, the stability threshold $1/3$ is obtained  for the Navier-Stokes equations by  introduce two new ideas: \textit{ (1) make the energy
estimates in short time scale $t\leq \kappa^{-\f16}$ and long time scale $t\geq \kappa^{-\f16}$ separately; (2) construct a new multiplier to control the growth caused by echo cascades.} In this paper, we further show that the approach is also applicable to Boussinesq equations with a Richardson number satisfying $\g^2>1/4$ and obtain the stability threshold for 2D Boussinesq  equations to $\alpha=\f13$. This stability threshold is consistent with the optimal stability  threshold for the 2D  Navier-Stokes  equations in Sobolev space.

We consider the velocity equation directly and define the energy as 
 \begin{align}\label{energy1}
 D(t):=\|m \hat{u}\|_{L^2_{k, \xi}}^2+\g^2 \|m\hat{\vth}\|_{L^2_{k, \xi}}^2 +\Re \sum_{k\in \mathbb{Z}}\int_{\mathbb{R}}  (m\hat{\vth}m \overline{\hat{u}}_1) d\xi,  
 \end{align}
then  $D(t)$ is a positive quadratic form when $\g^2>\f14$. Here multiplier 
\begin{equation}\label{eq: u newm}
 m(t, k, \xi)=\left\{\begin{array}{l}
|k_+, \xi-kt|^\frac{1}{2} \lan k, \xi\ran^s e^{M_0(t, k, \xi)},\ \   t\leq \kappa^{-\f16},\\
 |k_+, \xi-kt|^\frac{1}{2} \lan k, \xi\ran^s\mathcal{A} e^{\mathcal{M}_0(t, k, \xi)}, \ \ \   t\geq \kappa^{-\f16},
\end{array}\right.
\end{equation}
{are defined in Section 2 for $ t\leq \kappa^{-\f16}$ case to capture the inviscid damping  effect and for $ t\geq \kappa^{-\f16}$ case  to capture the inviscid damping and enhanced dissipation effect.}

In fact, mechanism leading to stability is the inviscid damping and enhanced dissipation effect.  Similar phenomena also happen in other fluids systems.  One can refer to \cite{BM,IJ,IJ1,MZ} for the inviscid damping effect of Euler equations, \cite{BGM2, BMV, CL, CW, M,WZ,WZ23} 
for the enhanced dissipation  effect of Navier-Stokes equations around Couette flow. For other shear flows, one may refer to \cite{IMM, WZZ, LWZ} for Kolmogorov flow and  \cite{ZEW,  DL} for Poiseuille flow. 
 
In  \cite{IJ1},  asymptotic stability of the Euler equations relies on multiplier bounds from symmetrization and different frequency considerations. 
In this paper, we   build the multiplier bounds that hold for all $k,l\in\mathbb{Z}$, reducing  the proof of nonlinear stability to a linear stability analysis in section 4 and section 6, and the estimate of the multiplier bounds section 3 and section 5. For the linear analysis,  when the Richardson number  $\g^2>1/4$, we identify a $\lan \na_L\ran^\f12$ energy structure \cite{RW} and a coupled energy, which will cancel the linear term and close the linear energy estimate.




\smallskip

{\bf{Notations:}}
Through this paper, we set $\pa_t=\f{\p}{\p t}$. We denote by $C$ a positive constant depending only on $ \g,s,\e$ (or some other fixed constants), but independent of $\nu,\mu,t$, which may be different from line to line.
 The notation
$A\lesssim B$ means that 
$A\leq CB $, $A\sim B$ means that $A\lesssim B$ and $B\lesssim A$. For $a,b \in \mathbb{R}$, set
\begin{align*}
|a,b|={\sqrt{a^2+b^2}}, \ \ \lan a\ran=\sqrt{1+a^2},\ \lan a,b \ran=\sqrt{1+a^2+b^2},\ \  (\text{then}\ \   \lan\na_L \ran=(1-\Delta _L )^\f12).
\end{align*}
We define the Fourier transform as
\begin{align*}
\mathcal{F}(f):=\hat{f}(k, \xi)=\f{1}{2\pi}\iint_{\mathbb{T}\times\mathbb{R}}e^{-i(kX+\xi Y)}f(X,Y)dXdY,
\end{align*}
and
$$
f(X, Y)=\f{1}{2\pi}\sum_{k\in\mathbb{Z}} \int_{\mathbb{R}} e^{i(kX+\xi Y)} \hat{f}(k, \xi) d\xi.
 $$
For $a\in \mathbb{R}$, we define the  inhomogeneous Sobolev space norm
\begin{align}\label{Hs}
\|f\|_{H^a}^2:=\|f\|_{H^a_{X,Y}}^2=\sum_{k\in\mathbb{Z}}\int_{\mathbb{R}}\lan k,\xi\ran^{2a}|\hat{f}(k, \xi)|^2d\xi,
\end{align}{and the operator}
\begin{align}
&\lan\na_L \ran^{a}f(X, Y):=\f{1}{2\pi}\sum_{k\in\mathbb{Z}} \int_{\mathbb{R}} e^{i(kX+\xi Y)} \lan k,\xi-kt\ran^a\hat{f}(k, \xi) d\xi.\nonumber
\end{align}


\smallskip
\setcounter{equation}{0}
\section{ Sketch of the proof of  Theorem \ref{th:main} }
\begin{lemma}\label{Lem4}
For fixed $\lambda>0$, we have (here $C$ depends only on $\lambda$)
\begin{equation}\label{Poisson}
   \int_\R \f{d\eta}{|a,\eta|^{1+\lambda}|b,z-\eta|^{1+\lambda}}\leq\f{C}{|ab|^{\lambda}}\f{|a+b|^{\lambda}}{|a+b,z|^{1+\lambda}},\quad \forall \  a>0,\ b>0,\ z\in\R.
\end{equation}
\end{lemma}
\begin{proof} Notice that
\begin{align*}
&\int_\R \f{|a+b,z|^{1+\lambda}d\eta}{|a,\eta|^{1+\lambda}|b,z-\eta|^{1+\lambda}}\leq 2^{\lambda}\int_\R \f{|a,\eta|^{1+\lambda}+|b,z-\eta|^{1+\lambda}}{|a,\eta|^{1+\lambda}|b,z-\eta|^{1+\lambda}}d\eta\\=&
2^{\lambda}\int_\R \f{d\eta}{|a,\eta|^{1+\lambda}}+2^{\lambda}\int_\R \f{d\eta}{|b,z-\eta|^{1+\lambda}}=\f{C}{|a|^{\lambda}}+\f{C}{|b|^{\lambda}}=C\f{|a|^{\lambda}+|b|^{\lambda}}{|ab|^{\lambda}}
\leq\f{C|a+b|^{\lambda}}{|ab|^{\lambda}},
\end{align*}
dividing both sides by $|a+b,z|^{1+\lambda}$ gives \eqref{Poisson}.
\end{proof}
The following lemma is used to prove the inviscid damping estimate in Theorem \ref{th:main}.
\begin{lemma}\label{Lem3}
    Let $s>3/2$, $t\geq0$. {Assume that $\div_L u=0$}. Then
    \begin{align*}
       &\|(u_1)_{\neq}\|_{L^2}+\lan t\ran\|u_2\|_{L^2}+\lan t\ran\|\hat{u}_2\|_{L^1}\leq C\lan t\ran^{-\f12}\|\lan \nabla_L\ran^{\f12} u_{\neq}\|_{H^s},\\& \|\vth_{\neq}\|_{L^2}\leq C\lan t\ran^{-\f12}\|\lan \nabla_L\ran^{\f12} \vth_{\neq}\|_{H^s}.
    \end{align*}
\end{lemma}
\begin{proof}
For  $k\in \mathbb{Z}\setminus\{0\}$, $\xi\in\R$, we have\begin{align*}
(1+k^2+\xi^2)(1+k^2+(\xi-kt)^2)\geq& k^2+(\xi^2+k^2)(k^2+(\xi-kt)^2)\geq k^2+(\xi k-k(\xi-kt))^2\\
=&k^2+k^4t^2\geq k^2(1+t^2)\geq 1+t^2.
 \end{align*}Thus 
\begin{align*}
&|\hat{u}_1|^2\leq \f{(1+k^2+\xi^2)^s(1+k^2+(\xi-kt)^2)^\f12 }{ (1+t^2)^\f12}|\hat{u}_1|^2,\\
&|\hat{\vth}|^2\leq \f{(1+k^2+\xi^2)^s(1+k^2+(\xi-kt)^2)^\f12 }{ (1+t^2)^\f12}|\hat{\vth}|^2.
 \end{align*}
and {(using $k \hat{u}_1+(\xi-kt) \hat{u}_2=0$ as $\div_L u=0$)}
\beno
|\hat{u}_2|^2=\f{k^2}{k^2+(\xi-kt)^2} |\hat{u}|^2\leq \f{(1+k^2+\xi^2)^s(1+k^2+(\xi-kt)^2)^\f12}{(1+t^2)^\f32} |\hat{u}|^2.
\eeno
{Then by Plancherel' s formula, we get} \begin{align*}
       &\|(u_1)_{\neq}\|_{L^2}+\lan t\ran\|u_2\|_{L^2}\leq C\lan t\ran^{-\f12}\|\lan \nabla_L\ran^{\f12} u_{\neq}\|_{H^s},\quad  \|\vth_{\neq}\|_{L^2}\leq C\lan t\ran^{-\f12}\|\lan \nabla_L\ran^{\f12} \vth_{\neq}\|_{H^s}.
    \end{align*}
For the $L^1$ part, since $P_0 u_2=0$, we have
\begin{align*}
&\|\hat{u}_2\|_{L^1}=\sum_{k\in \mathbb{Z}\setminus\{0\}}\int_{\mathbb{R}}|\hat{u}_2| d\xi=\sum_{k\in \mathbb{Z}\setminus\{0\}}\int_{\mathbb{R}}\f{|k|}{(k^2+(\xi-kt)^2)^\f12}|\hat{u}| d\xi\non\\
\leq &\Big(\sum_{k\in \mathbb{Z}\setminus\{0\}}\int_{\mathbb{R}}(k^2+(\xi-kt)^2)^\f12(k^2+\xi^2)^s|\hat{u}|^2 d\xi\Big)^\f12\times\\&\Big(\sum_{k\in \mathbb{Z}\setminus\{0\}}\int_{\mathbb{R}}\f{|k|^2}{(k^2+\xi^2)^s(k^2+(\xi-kt)^2)^\f32} d\xi\Big)^\f12\non\\
\leq& \|\lan \na _L\ran^\f12 u_{\neq}\|_{H^s}\Big(\sum_{k\in \mathbb{Z}\setminus\{0\}} \f{1}{(k^2)^{s-\f32}} \int_\R \f{|k|^2}{(k^2+\xi^2)^{\f32}(k^2+(\xi-kt)^2)^{\f32}}d\xi\Big)^\f12\non\\
\leq& \|\lan \na _L\ran^\f12 u_{\neq}\|_{H^s}\big(\sum_{k\in \mathbb{Z}\setminus\{0\}} \f{1}{(k^2)^{s-\f32}} \f{C}{|2k,kt|^{3}}\Big)^\f12\leq C\lan t\ran^{-\f32}\|\lan \nabla_L\ran^{\f12} u_{\neq}\|_{H^s},
\end{align*}
where we used Lemma \ref{Lem4} (with $a=b=k,  z=kt, \lambda =2$) to
{deduce that}
\begin{equation*}
   \int_\R \f{|k|^{2}d\xi}{|k,\xi|^{3}|k, \xi-kt|^{3}}\leq\f{C}{|2k,kt|^{3}}.
\end{equation*}
\end{proof}
When $t\leq \kappa^{-\f16}$, we will use the following key proposition.\begin{proposition}\label{pr01}
If $\kappa\in(0, 1)$, $\g>\f12$, $s>\f32$ and $\f{\nu+\mu}{2\g\sqrt{\nu \mu} }< 2-\e$, $0<\e<1/\g$, then there exist constants $c_1>0$, $\widetilde C>0 $ (depending only on $\g, s, \e$) such that if $
\| (u_{in}, \vth_{in})\|_{H^{s+1/2}}\leq c_1\kappa^\f13,
$ then it holds that 
$\| \lan\na_L\ran^\f12 (u, \vth)(t)\|_{H^s }\le  {\widetilde C}\| (u_{in}, \vth_{in})\|_{H^{s+1/2}}$ for $0\leq t\leq T_0=\kappa^{-\f16}$.
\end{proposition}
The key point is to show that $\f{d}{dt}E(t)\leq C\lan t\ran\|A_k(\hat{u},\hat{\vth})\|_{{L_{k,\xi}^2}}^3 $, see Proposition \ref{pr1}. Here 
\begin{align*}
&E(t):=\sum_{k\in \mathbb{Z}}\int_{\mathbb{R}} \Big(\big|A_k\hat{u}\big|^2+\g^2\big|A_k\hat{\vth}\big|^2+\Re (A_k\hat{\vth}A_k \overline{\hat{u}}_1)\Big) d\xi \sim \|A_k(\hat{u},\hat{\vth})\|_{{L_{k,\xi}^2}}^2\sim\| \lan\na_L\ran^\f12 (u, \vth)\|_{H^s }^2,\\
&A_k(t, \xi):=|k_+, \xi-kt|^\frac{1}{2} \lan k, \xi\ran^s e^{M_0(t, k, \xi)},\quad k_+=\max \{|k|, 1\},\\
&M_0(t, k, \xi):=-\int^t_0 \frac{C_{\g}k^2}{k^2+(\xi-k\tau)^2 } d\tau,\quad C_{\g}=\max\Big\{1,\f{2}{2\g-1}\Big\}.
\end{align*}

When $t\geq \kappa^{-\f16}$, we define the Multiplier $\mathcal{M}(t, k, \xi)$ as 
\begin{align*}
\mathcal{M}(t, k, \xi)= |k_+, \xi-kt|^\frac{1}{2} \lan k, \xi\ran^s\mathcal{A} e^{\mathcal{M}_0}, \ \text{where} \ \mathcal{M}_0=\mathcal{M}_1+\mathcal{M}_2+\mathcal{M}_3,
\end{align*}
with
\begin{align*}
&\mathcal{A}(t, k, \xi)=  e^{\epsilon\kappa^\f13 t\mathbf{1}_{k\neq 0}+\kappa^{-\f13} t^{-2}}, \quad \epsilon=\e/16, \\
&\mathcal{M}_1(t, k, \xi)={\psi_{1}}\big(\kappa^{\f13}|k|^{\f23}(\xi/k-t)\big),\ \text{for} \  k\neq 0,\quad \mathcal{M}_1(t, 0, \xi)
=0,\\
&\mathcal{M}_2(t,k, \xi)=C_{\g}\psi_{1-\delta}(\xi/k-t),\ \text{for} \  k\neq 0,\quad \mathcal{M}_2(t, 0, \xi)=0,\\
&\mathcal{M}_3(t, k, \xi)=\sum_{j\in \mathbb{Z}\setminus\{0\}}
j^{-1}\lan k-j\ran^{-\delta}\psi_{\delta}\Big(\f{\xi-tj}{\lan k-j\ran+|j|}\Big).
&\end{align*} 
Where 
for  $ \lambda>0$, $\psi_{\lambda}$ solves $\psi_{\lambda}'(x)=\lan x\ran^{-1-\lambda}$, $\psi_{\lambda}(0)=0$ (then $ \psi_{1}=\arctan$). We fix $ \delta$ such that $0<\delta<\min(s-3/2,1/2)$. 
Then $\psi_{\lambda},\ \mathcal{M}_1,\ \mathcal{M}_2,\ \mathcal{M}_3,\ \mathcal{M}_0 $ are bounded functions, and
\begin{align}\nonumber
 \Upsilon(t,k,\xi):&=-\pa_t\mathcal{M}_3(t,k,\xi)=\sum_{j\in \mathbb{Z}\setminus\{0\}}
\f{\lan k-j\ran^{-\delta}(\lan k-j\ran+|j|)^{\delta}}{|\lan k-j\ran+|j|,\xi-tj|^{1+\delta}}\\ \label{Up}
&=\sum_{l\in \mathbb{Z}\setminus\{k\}}
\f{\lan l\ran^{-\delta}(\lan l\ran+|k-l|)^{\delta}}{|\lan l\ran+|k-l|,\xi-t(k-l)|^{1+\delta}}.
\end{align}
Thus $0<\kappa^{-\f13} t^{-2}\leq 1$ and 
$\mathcal{M}(t, k, \xi)\sim |k_+, \xi-kt|^\frac{1}{2} \lan k, \xi\ran^s\mathcal{A}\sim \lan k, \xi-kt\ran^\frac{1}{2} \lan k, \xi\ran^se^{\epsilon\kappa^\f13 t\mathbf{1}_{k\neq 0}}$ for $t\geq T_0=\kappa^{-\f16}$. 

We will use the following key proposition.
\begin{proposition}\label{pr02}
If $\kappa\in(0, 1)$, $\g>\f12$, $s>\f32$ and $\f{\nu+\mu}{2\g\sqrt{\nu \mu} }< 2-\e$, $0<\e<1/\g$. Let\begin{align}\label{E*t}
&E_*(t):=\sum_{k\in \mathbb{Z}}\int_{\mathbb{R}} \Big(\big|\mathcal{M}\hat{u}\big|^2+\g^2\big|\mathcal{M}\hat{\vth}\big|^2+\Re (\mathcal{M}\hat{\vth}\mathcal{M} \overline{\hat{u}}_1)\Big) d\xi.
\end{align} Then there exists a constant $c_2>0$ (depending only on $\g, s, \e$) such that if $t\geq T_0=\kappa^{-\f16}$ $
\|\mathcal{M}(\hat{u}, \hat{\vth})(t)\|_{L^2_{k, \xi}}\leq c_2\kappa^\f13,
$ then it holds that 
$\f{d}{dt}E_*(t)\leq0$.
\end{proposition}The key point is to show that (see Proposition \ref{prop: long-time i}, here $\e_1=\e/4$)
\begin{align}\label{continuity}
        &\f{d}{dt}E_*(t)+\e_1(\|\lan \xi/k-t\ran^{\f{\delta}{2}}\mathcal{M}\hat{u}_2\|_{L_{k,\xi}^2}^2+ \kappa^\f13\||k|^{\f13}\mathcal{M}(\hat{u},\hat{\vth})\|_{L_{k,\xi}^2}^2+\|\sqrt{\Upsilon}\mathcal{M}(\hat{u},\hat{\vth})\|_{L_{k,\xi}^2}^2\non\\
       &+\kappa^{-\f13} t^{-3}  \|\mathcal{M} (\hat{u}, \hat{\vth})\|_{L_{k,\xi}^2}^2
       + \nu\||k,\xi-kt|\mathcal{M}(\hat{u},\hat{\vth})\|_{L_{k,\xi}^2}^2)\non\\
          \leq &C\|\mathcal{M}(\hat{u},\hat{\vth})\|_{{L_{k,\xi}^2}}(\||k|^{\f13}\mathcal{M}(\hat{u},\hat{\vth})\|_{L_{k,\xi}^2}^2+\kappa^{\f23}\||k,\xi-kt|\mathcal{M}(\hat{u},\hat{\vth})\|_{L_{k,\xi}^2}^2+\kappa^{-\f23}t^{-3}\|\mathcal{M}\hat{u}\|_{L_{k,\xi}^2}^2\non\\
          &+ \kappa^{-\f13}\|\lan \xi/k-t\ran^{\f{\delta}{2}}\mathcal{M}\hat{u}_2\|_{L_{k,\xi}^2}^2+\kappa^{-\f13}\|\sqrt{\Upsilon}\mathcal{M}(\hat{u},\hat{\vth})\|_{L_{k,\xi}^2}^2).
\end{align} Now we are in a position to prove Theorem 1.1, by implementing Propositions \ref{pr01}
 and \ref{pr02}.
\begin{proof}
If $\|(u_{in}, \vth_{in})\|_{H^{s+1/2}}\leq c \kappa^{\f13}\leq c_1 \kappa^{\f13}$.
{ By Propositions \ref{pr01}, 
\begin{align}\label{T<1}\| \lan\na_L\ran^\f12 (u, \vth)(t)\|_{H^s }\le  \widetilde C\| (u_{in}, \vth_{in})\|_{H^{s+1/2}}\leq   \widetilde C c \kappa^{\f13}, \ \text{for} \  0\leq t\leq T_0=\kappa^{-\f16}.\end{align}
As $\mathcal{M}(t, k, \xi)\sim \lan k, \xi-kt\ran^\frac{1}{2} \lan k, \xi\ran^se^{\epsilon\kappa^\f13 t\mathbf{1}_{k\neq 0}}$ for $t\geq T_0=\kappa^{-\f16}$, we have
\begin{align}\label{T<}
\|\mathcal{M} (u, \vth)(t)\|_{L^2}\leq& C_* e^{\epsilon \kappa^\f13 t} \|\lan\na_L\ran^\f12 (u, \vth)(t)\|_{H^s}
\leq 2C_*\widetilde C\|(u_{in}, \vth_{in})\|_{H^{s+1/2}}\non\\ \leq& 2C_*\widetilde C  c\kappa^{\f13},\quad 
\text{for}\ t= T_0=\kappa^{-\f16},\end{align}
where we used that $\epsilon \kappa^\f13 T_0=\epsilon \kappa^\f16\leq \epsilon=\f{\e}{16}\leq\ln 2$. 
Thus if $0<c\leq\min\{c_1, c_2/(4C_*\widetilde C)\}$ then $\|\mathcal{M} (u, \vth)(T_0)\|_{L^2}\leq c_2\kappa^\f13/2$.

For $t\geq T_0$, we use the bootstrap argument. Assume 
$$ \|\mathcal{M}(u, \vth)(t)\|_{L^2}\leq c_2\kappa^\f13,\   \text{for} \  T_0\leq t\leq T,$$
by Propositions \ref{pr02}, we have $E_*(t)\leq E_*(T_0)$  for $ T_0\leq t\leq T$. As $\g>\f12$, by \eqref{E*t} we have
$$
C_{1,\g}^{-1} \|\mathcal{M} (u, \vth)(t)\|_{L^2}^2\leq E_*(t)\leq C_{1,\g} \|\mathcal{M} (u, \vth)(t)\|_{L^2}^2,
$$
where $ C_{1,\g}>1$ is a constant depending only on $\g$. Then
\begin{align*}
\|\mathcal{M} (u, \vth)(t)\|^2_{L^2}&\leq C_{1,\g} E_*(t)\leq C_{1,\g}  E_*(T_0)\leq C_{1,\g}^2 \|\mathcal{M} (u, \vth)\|_{L^2}^2\Big|_{t=T_0}\\
&\leq C_{1,\g}^2 (2C_*\widetilde Cc\kappa^{\f13})^2\leq ({c_2}\kappa^{\f13}/2)^2, \ 
 \text{for} \  T_0\leq t\leq T.
 \end{align*}
where $c=\min\{c_1, {c_2}/({4C_*\widetilde C C_{1,\g}}) \}$.}
 
Then the bootstrap argument implies $T=+\infty$, and
$$
\|\mathcal{M} (u, \vth)(t)\|_{L^2}\leq C_{1,\g}\|\mathcal{M} (u, \vth)\|_{L^2}\Big|_{t=T_0}, \quad 
 \text{for} \  t\geq T_0,$$
which along with \eqref{T<} and $\mathcal{M}(t, k, \xi)\sim \lan k, \xi-kt\ran^\frac{1}{2} \lan k, \xi\ran^se^{\epsilon\kappa^\f13 t\mathbf{1}_{k\neq 0}}$ for $t\geq T_0$ gives that
\begin{align}\non
&e^{\epsilon\kappa^\f13 t}\|\lan \na_L\ran^\f12 (u_{\neq}, \vth_{\neq})(t)\|_{H^s}+\|\lan \na_L\ran^\f12 (u, \vth)(t)\|_{H^s}
\leq C\|\mathcal{M} (u, \vth)(t)\|_{L^2}\\ \label{T>}&\leq C\|(u_{in}, \vth_{in})\|_{H^{s+1/2}}, \quad 
 \text{for} \  t\geq T_0. \end{align}
 Then the global stability estimate \eqref{stability} and the enhanced dissipation  estimate \eqref{3} follows from \eqref{T<1}, \eqref{T>} and 
 $\epsilon \kappa^\f13 T_0\leq\ln 2$. Moreover the  inviscid damping estimate \eqref{2} follows from the enhanced dissipation  estimate \eqref{3} and Lemma \ref{Lem3}.
          \end{proof}

\smallskip

\section{   Property of the multiplier $A_k$ in small time scale $t<\kappa^{-\frac{1}{6}}$}
\setcounter{equation}{0}
Recall the multiplier used when $t\leq T_0= \kappa ^{-\f16}$:
\begin{align*}
 A_k(t, \xi)&=|k_+, \xi-kt|^\frac{1}{2} \lan k, \xi\ran^s e^{M_0(t, k, \xi)},\quad k_+=\max \{|k|, 1\},
\\
M_0(t, k, \xi)&=-\int^t_0 \frac{C_{\g}k^2}{k^2+(\xi-k\tau)^2 } d\tau,
\end{align*}
then $M_0(t, k, \xi) $ is well defined and smooth for $ (k,\xi)\in\R^2\setminus\{(0,0)\}$. 
We notice that  $-\pi C_{\g}\leq M_0(t, k, \xi)\leq0$, then 
$A_k(t, \xi)\sim|k_+, \xi-kt|^\frac{1}{2} \lan k, \xi\ran^s \sim\lan k, \xi-kt\ran^\frac{1}{2} \lan k, \xi\ran^s$.
 \begin{lemma}\label{lem3.2}
If $(k,\xi)\in\R^2\setminus\{(0,0)\}$, $t\geq0$ then
\begin{align}\label{M0a}
|\p_k M_0(t, k, \xi), \p_\xi M_0(t, k, \xi)|\leq \f{ C\lan t\ran }{| k, \xi-kt|}.
\end{align}
\end{lemma}\begin{proof} 
For $k\neq0$ we have\begin{align*}
M_0(t, k, \xi)&=C_\g(\arctan({\xi}/{k}-t)-\arctan({\xi}/{k})),
\end{align*}
then
\begin{align*}
&|\p_k M_0(t, k, \xi), \p_\xi M_0(t, k, \xi)|= C_{\g}\Big| \f{\xi}{k^2+\xi^2}-\f{\xi}{k^2+(\xi-kt)^2},\f{k}{k^2+\xi^2}-\f{k}{k^2+(\xi-kt)^2}\Big|\\
= & C_\g\f{|k,\xi||(\xi-kt)^2-\xi^2| }{(k^2+\xi^2)(k^2+(\xi-kt)^2)}
=  C_\g |k,\xi|\f{t|\xi k+k(\xi-kt)|}{|k,\xi|^2| k, \xi-kt|^2}\\
\leq & C_\g |k,\xi|\f{t|\xi,k||k,\xi-kt|}{|k,\xi|^2| k, \xi-kt|^2}
= \f{ C_\g t}{| k, \xi-kt|},
\end{align*}which implies \eqref{M0a}. The case $k=0$ follows by taking limit.
\end{proof}
 \begin{lemma}\label{lem3.1}
If $(k, \xi),(l, \eta)\in\mathbb{Z}\times\R$, $t\geq0$ then
\begin{align}\label{M0}
|M_0(t, k, \xi)-M_0(t, l, \eta)|\leq  C\f{ \lan t\ran |k-l, \xi-\eta|}{\lan k, \xi-kt\ran+\lan l, \eta-lt\ran}.
\end{align}
\end{lemma}
\begin{proof}
If $(k,l)=(0,0)$ then $M_0(t, k, \xi)=M_0(t, l, \eta)=0$, \eqref{M0} is clearly true. Now we assume $(k,l)\neq(0,0)$ then 
$$\lan k, \xi-kt\ran+\lan l, \eta-lt\ran\leq3( | k, \xi-kt|+| l, \eta-lt|).$$

\textbf{Case 1:}  
$| k-l,\xi-\eta-(k-l)t|\geq (| k, \xi-kt|+| l, \eta-lt|)/4$. 
In this case\begin{align*}
| k, \xi-kt|+| l, \eta-lt|\leq 4| k-l,\xi-\eta-(k-l)t|\leq  4( t+1)\ran | k-l,\xi-\eta|.
\end{align*} Then, by $-\pi C_{\g}\leq M_0(t, k, \xi)\leq0$, we have
\begin{align*}
|M_0(t, k, \xi)-M_0(t, l, \eta)|\leq { \pi} C_{\g}\leq C\f{ \lan t\ran |k-l, \xi-\eta|}{| k, \xi-kt|+| l, \eta-lt|}\leq C\f{ \lan t\ran |k-l, \xi-\eta|}{\lan k, \xi-kt\ran+\lan l, \eta-lt\ran}.
\end{align*}

\textbf{Case 2:}  
$| k-l,\xi-\eta-(k-l)t|\leq (| k, \xi-kt|+|l, \eta-lt|)/4$. By Taylor's formula,\begin{align*}
&|M_0(t, k, \xi)-M_0(t, l, \eta)|\\
\leq& \int_0^1|\p_k M_0(t, k-\tau(k-l), \xi-\tau(\xi-\eta)),\p_\xi M_0(t, k-\tau(k-l), \xi-\tau(\xi-\eta))||k-l, \xi-\eta|d\tau\\
\leq& \int_0^1\f{ C\lan t\ran |k-l, \xi-\eta|}{| k-\tau(k-l), \xi-\tau(\xi-\eta)-(k-\tau(k-l))t|}d\tau\leq C\f{ \lan t\ran |k-l, \xi-\eta|}{\lan k, \xi-kt\ran+\lan l, \eta-lt\ran}.
\end{align*}
Here we used that (for $\tau\in[0,1]$)
\begin{align*}
&| k-\tau(k-l), \xi-\tau(\xi-\eta)-(k-\tau(k-l))t|\geq | k, \xi-kt|-\tau| k-l,\xi-\eta-(k-l)t|,\\
&| k-\tau(k-l), \xi-\tau(\xi-\eta)-(k-\tau(k-l))t|\geq |l, \eta-lt|-(1-\tau)| k-l,\xi-\eta-(k-l)t|,\end{align*}
and that (adding the above 2 inequalities)
\begin{align*}
&2| k-\tau(k-l), \xi-\tau(\xi-\eta)-(k-\tau(k-l))t|\\ 
\geq& | k, \xi-kt|+|l, \eta-lt|-| k-l,\xi-\eta-(k-l)t|\\ \geq&| k, \xi-kt|+|l, \eta-lt|-(| k, \xi-kt|+|l, \eta-lt|)/4\\=&
3(| k, \xi-kt|+|l, \eta-lt|)/4\geq(\lan k, \xi-kt\ran+\lan l, \eta-lt\ran)/4.
\end{align*}
\end{proof}
%
%

\begin{lemma}\label{Ak}If $(k, \xi),(l, \eta)\in\mathbb{Z}\times\R$, $t\geq0$ then
\begin{align}\label{Ak1}
&|l A_k^2(t, \xi)-kA_l^2(t, \eta)|+|(\eta-lt) A_k^2(t, \xi)-(\xi-kt) A_l^2(t, \eta)|\non\\
\leq& C\lan t\ran A_k(t, \xi) A_l(t, \eta)A_{k-l}(t, \xi-\eta) (\lan k, \xi\ran^{\frac{1}{2}-s}+\lan l, \eta\ran^{\frac{1}{2}-s}+\lan k-l, \xi-\eta\ran^{\frac{1}{2}-s}\non\\&+\lan k-l, (\xi-\eta)-(k-l)t\ran^{\frac{1}{2}-s}),
\end{align}
and
\begin{align}\label{Ak2}
&\lan l,\eta-lt \ran A_k^2 (t, \xi)\f{|k||k-l|}{k^2+(\xi-kt)^2}+\lan k,\xi-kt \ran A_l^2 (t, \eta)\f{|l||k-l|}{l^2+(\eta-lt)^2}\non\\
\leq& C\lan t\ran A_k(t, \xi) A_l(t, \eta)A_{k-l}(t, \xi-\eta)(\lan k, \xi\ran^{\frac{1}{2}-s}+\lan l, \eta\ran^{\frac{1}{2}-s}+\lan k-l, \xi-\eta\ran^{\frac{1}{2}-s}).
\end{align}
\end{lemma}
\begin{proof}

We first proof \eqref{Ak1}. Recall that $A_k(t, \xi)\sim\sim\lan k, \xi-kt\ran^\frac{1}{2} \lan k, \xi\ran^s$.

\textbf{Case 1:}  
$\langle k-l,\xi-\eta\rangle\geq (\langle k,\xi\rangle+\langle l,\eta\rangle)/4$.
In  this case, 
 \begin{align*}
&\f{\lan l, \eta-lt\ran A_k(t, \xi)}{A_l(t, \eta)A_{k-l}(t, \xi-\eta)}\\
\leq&C \f{\lan k, \xi-kt\ran^\frac{1}{2} }{\lan l, \eta-lt\ran^\frac{1}{2}\lan k-l, (\xi-\eta)-(k-l)t\ran^\frac{1}{2} } \f{\lan l, \eta-lt\ran\lan k, \xi\ran^s}{\lan l, \eta\ran^s\lan k-l, \xi-\eta\ran^s}\\
\leq&C \f{\lan l, \eta-lt\ran^\frac{1}{2}+\lan k-l, (\xi-\eta)-(k-l)t\ran^\frac{1}{2} }{\lan l, \eta-lt\ran^\frac{1}{2}\lan k-l, (\xi-\eta)-(k-l)t\ran^\frac{1}{2} } \f{\lan l, \eta-lt\ran}{\lan l, \eta\ran^s}\\
=&\f{C\lan l, \eta-lt\ran}{\lan k-l, (\xi-\eta)-(k-l)t\ran^\frac{1}{2} \lan l, \eta\ran^{s}}+\f{C\lan l, \eta-lt\ran^\frac{1}{2}}{\lan l, \eta\ran^{s}}
\\ \leq&\f{C\lan t\ran\lan l, \eta\ran}{\lan k-l, (\xi-\eta)-(k-l)t\ran^\frac{1}{2} \lan l, \eta\ran^{s}}+\f{C\lan t\ran\lan l, \eta\ran^\frac{1}{2}}{\lan l, \eta\ran^{s}}\\
=&\f{C\lan t\ran}{\lan k-l, (\xi-\eta)-(k-l)t\ran^\frac{1}{2} \lan l, \eta\ran^{s-1}}+\f{C\lan t\ran}{\lan l, \eta\ran^{s-\frac{1}{2}}}\\
\leq& C\lan t\ran(\lan l, \eta\ran^{\frac{1}{2}-s}+\lan k-l, (\xi-\eta)-(k-l)t\ran^{\frac{1}{2}-s}),
\end{align*}
and similarly 
\begin{align*}
&\f{\lan k, \xi-kt\ran A_l(t, \eta)}{A_k(t, \xi)A_{k-l}(t, \xi-\eta)}
\leq C\lan t\ran(\lan k, \xi\ran^{\frac{1}{2}-s}+\lan k-l, (\xi-\eta)-(k-l)t\ran^{\frac{1}{2}-s}),
\end{align*}
we  can get  
\begin{align*}
&|l A_k^2(t, \xi)-kA_l^2(t, \eta)|+|(\eta-lt)A_k^2(t, \xi)-(\xi-kt)A_l^2(t, \eta)|\\
\leq& 2|\lan l, \eta-lt\ran A_k^2(t, \xi)|+2|\lan k, \xi-kt\ran A_l^2(t, \xi)|\\
=& \left|\f{2\lan l, \eta-lt\ran A_k(t, \xi)}{A_l(t, \eta)A_{k-l}(t, \xi-\eta)}+\f{2\lan k, \xi-kt\ran A_l(t, \eta)}{A_k(t, \xi)A_{k-l}(t, \xi-\eta)}\right|A_k(t, \xi) A_l(t, \eta)A_{k-l}(t, \xi-\eta) \\
\leq& C\lan t\ran A_k(t, \xi) A_l(t, \eta)A_{k-l}(t, \xi-\eta) (\lan k, \xi\ran^{\frac{1}{2}-s}+\lan l, \eta\ran^{\frac{1}{2}-s}+\lan k-l, (\xi-\eta)-(k-l)t\ran^{\frac{1}{2}-s})
\end{align*}
without using symmetry.
\if0\begin{align*}
l_+\f{A_k(t, \xi)}{A_l(t, \eta)}&=l_+\f{|k_+, \xi-kt|^\frac{1}{2} }{|l_+, \eta-lt|^\frac{1}{2} } \f{\lan k, \xi\ran^s}{\lan l, \eta\ran^s} e^{M_0(t, k, \xi)-M_0(t, l, \eta)}\\
&\leq \lan t\ran A_{k-l}(t, \xi-\eta) (\lan k, \xi\ran^{\frac{1}{2}-s}+\lan l, \eta\ran^{\frac{1}{2}-s}+\lan k-l, \xi-\eta\ran^{\frac{1}{2}-s}+\lan k-l, (\xi-\eta)-(k-l)t\ran^{\frac{1}{2}-s}).
\end{align*}
where we use
\begin{align*}
l_+\f{|k_+, \xi-kt|^\frac{1}{2} }{|l_+, \eta-lt|^\frac{1}{2} }\leq|l_+, \eta-lt|^\frac{1}{2}|k_+, \xi-kt|^\frac{1}{2}\lesssim|l_+, \eta-lt|^\frac{1}{2}|(k-l)_+, (\xi-\eta)-(k-l)t|^\f12,
\end{align*}
and
\begin{align*}
\lan k, \xi\ran^{s}\lesssim \lan k-l, \xi-\eta\ran^{s},
\end{align*}
and 
\begin{align*}
e^{|M_0(t, k, \xi)-M_0(t, l, \eta)|}\leq e^{ \frac{\lan t\ran |k-l, \xi-\eta|}{\lan k, \xi-kt\ran+\lan l, \eta-lt\ran}}\leq e^{\f{(k-l)t^2}{(k-l)^2+((\xi-\eta)-(k-l)\tau)^2} }= e^{|M_0(t, k-l, \xi-\eta)|}.
\end{align*}
\textcolor{red}{
\begin{align*}
|M_0(t, k, \xi)-M_0(t, l, \eta)|
&\leq \lan t\ran  |\frac{k^2}{k^2+(\xi-k t)^2 } - \frac{l^2}{l^2+(\eta-l t)^2 } |\\
&\leq \lan t\ran  |\frac{k}{\sqrt{k^2+(\xi-k t)^2} } - \frac{l}{\sqrt{l^2+(\eta-l t)^2} } |,
\end{align*}
where we use
$$
f(x) := \frac{x^2}{x^2+a^2}, \ g(x) := \frac{x}{\sqrt{x^2+a^2}},\  f'(x) = \frac{2ax^2}{(x^2+a^2)^2}\leq g'(x) = \frac{a^2}{(x^2+a^2)^{\frac{3}{2}}},\  f(0)=g(0), \ f(x) \leq g(x),
$$
then
\begin{align*}
\left|\frac{k^2}{k^2+(\xi-k t)^2} - \frac{l^2}{l^2+(\eta-l t)^2}\right| &= |f(k) - f(l)| \leq |g(k) - g(l)| = \left|\frac{k}{\sqrt{k^2+(\xi-k t)^2}} - \frac{l}{\sqrt{l^2+(\eta-l t)^2}}\right|\\
&\leq \left|\frac{1}{\sqrt{k^2+(\xi-k t)^2}} - \frac{1}{\sqrt{l^2+(\eta-l t)^2}}\right||k-l|,
\end{align*}
and
\beno
|\f{1}{a}-\f{1}{b}|\leq \f{1}{a+b}, \ a, b>0.
\eeno}\fi

\textbf{Case 2:}  
$\langle k-l,\xi-\eta\rangle\leq (\langle k,\xi\rangle+\langle l,\eta\rangle)/4$.
(In this case $\langle k,\xi\rangle\sim\langle l,\eta\rangle $. By symmetry,  we can assume that 
$\lan k, \xi-kt\ran\leq\lan l, \eta-lt\ran$). 


Recall that $A_k(t, \xi)=|k_+, \xi-kt|^\frac{1}{2} \lan k, \xi\ran^s e^{M_0(t, k, \xi)}$, $-\pi C_{\g}\leq M_0(t, k, \xi)\leq0$. We divide
\begin{align}\label{T}
|l A_k^2(t, \xi)-kA_l^2(t, \eta)|
&=\Big| l|k_+, \xi-kt| \lan k, \xi\ran^{2s} e^{2M_0(t, k, \xi)}-k |l_+, \eta-lt| \lan l, \eta\ran^{2s} e^{2M_0(t, l, \eta)}\Big|\non\\
&\leq \cT_1+\cT_2+\cT_3,
\end{align}
where 
\begin{align*}
&\cT_1=|k_+, \xi-kt|\lan k, \xi\ran^{2s}\Big| l e^{2 M_0 (t, k, \xi)}-k e ^{2 M_0 (t, l, \eta)}\Big|,\\
&\cT_2=\Big||k_+, \xi-kt| -|l_+, \eta-lt|\Big| \lan k, \xi\ran^{2s} |k|  e ^{2 M_0 (t, l, \eta)},\\
&\cT_3=|k|e ^{2 M_0 (t, l, \eta)}|l_+, \eta-lt| \Big|\lan l, \eta\ran^{2s}- \lan k, \xi\ran^{2s}\Big|.
\end{align*}
Then we have
\begin{align*}
\cT_1&=|k_+, \xi-kt|\lan k, \xi\ran^{2s} \Big| l e^{2 M_0 (t, k, \xi)}-k e^{2 M_0 (t, l, \eta)}\Big|\\
&\leq |k_+, \xi-kt|\lan k, \xi\ran^{2s}(|k-l|e^{2 M_0 (t, l, \eta)}+|l||e^{2 M_0 (t, k, \xi)}-e^{2 M_0 (t, l, \eta)}|)\\
&\leq C\lan k, \xi-kt\ran\lan k, \xi\ran^{2s}(|k-l|+|l||M_0 (t, k, \xi)-M_0 (t, l, \eta)|),\end{align*}
and
\begin{align*}
\cT_2=\Big||k_+, \xi-kt| -|l_+, \eta-lt|\Big| \lan k, \xi\ran^{2s} |k|  e^{2 M_0 (t, l, \eta)}
\leq C\lan k-l, (\xi-\eta)-(k-l)t\ran \lan k, \xi\ran^{2s} |k|,
\end{align*}and\begin{align*}
 \cT_3&=|k|e ^{2 M_0 (t, l, \eta)}|l_+, \eta-lt| \Big|\lan l, \eta\ran^{2s}- \lan k, \xi\ran^{2s}\Big|\\
&\leq C |k|\lan l, \eta-lt\ran\Big|\lan l, \eta\ran^{2s}- \lan k, \xi\ran^{2s}\Big|\\
&\leq C|k|\lan l, \eta-lt\ran\lan k-l, \xi-\eta\ran\lan k, \xi\ran^{2s-1}.
\end{align*}
Similarly\begin{align}\label{T'}
&|(\eta-lt)A_k^2(t, \xi)-(\xi-kt)A_l^2(t, \eta)|\non\\
=&\Big|(\eta-lt)|k_+, \xi-kt| \lan k, \xi\ran^{2s} e^{2M_0(t, k, \xi)}-(\xi-kt) |l_+, \eta-lt| \lan l, \eta\ran^{2s} e^{2M_0(t, l, \eta)}\Big|\non\\
\leq& \cT_1'+\cT_2'+\cT_3',
\end{align}
where 
\begin{align*}
&\cT_1'=|k_+, \xi-kt|\lan k, \xi\ran^{2s}\Big| (\eta-lt) e^{2 M_0 (t, k, \xi)}-(\xi-kt) e ^{2 M_0 (t, l, \eta)}\Big|,\\
&\cT_2'=\Big||k_+, \xi-kt| -|l_+, \eta-lt|\Big| \lan k, \xi\ran^{2s} |\xi-kt|  e ^{2 M_0 (t, l, \eta)},\\
&\cT_3'=|\xi-kt|e ^{2 M_0 (t, l, \eta)}|l_+, \eta-lt| \Big|\lan l, \eta\ran^{2s}- \lan k, \xi\ran^{2s}\Big|,
\end{align*}
and
\begin{align*}
\cT_1'
&\leq C\lan k, \xi-kt\ran\lan k, \xi\ran^{2s}(|(\xi-\eta)-(k-l)t|+|\eta-lt||M_0 (t, k, \xi)-M_0 (t, l, \eta)|),
\end{align*}
\begin{align*}
\cT_2'
&\lesssim\lan k-l, (\xi-\eta)-(k-l)t\ran \lan k, \xi\ran^{2s} |\xi-kt|,
\end{align*}\begin{align*}
 \cT_3'
&\lesssim|\xi-kt|\lan l, \eta-lt\ran\lan k-l, \xi-\eta\ran\lan k, \xi\ran^{2s-1} ,
\end{align*}
Summing up  \eqref{T} and \eqref{T'}, we have
 \begin{align*}
&|l A_k^2(t, \xi)-kA_l^2(t, \eta)|+|(\eta-lt) A_k^2(t, \xi)-(\xi-kt) A_l^2(t, \eta)|\\
\leq& C\lan k, \xi-kt\ran\lan k, \xi\ran^{2s}(\lan k-l, (\xi-\eta)-(k-l)t\ran+\lan l, \eta-lt\ran|M_0 (t, k, \xi)-M_0 (t, l, \eta)|)\\
& 
+C\lan k, \xi-kt\ran\lan l, \eta-lt\ran\lan k-l, \xi-\eta\ran\lan k, \xi\ran^{2s-1}.
\end{align*}
By Lemma \ref{lem3.1}, we have
\begin{align*}
&\f{|l A_k^2(t, \xi)-kA_l^2(t, \eta)|+|(\eta-lt) A_k^2(t, \xi)-(\xi-kt) A_l^2(t, \eta)|}{\lan k, \xi-kt\ran^{\f12}\lan k, \xi\ran^{s}\lan l, \eta-lt\ran^{\f12}\lan l, \eta\ran^{s}\lan k-l, (\xi-\eta)-(k-l)t\ran^{\f12}\lan k-l, \xi-\eta\ran^{s}}\\
\leq& \f{C\lan k, \xi-kt\ran^{\f12}\lan k-l, (\xi-\eta)-(k-l)t\ran^{\f12}}{\lan l, \eta-lt\ran^{\f12}\lan k-l, \xi-\eta\ran^{s}}\\
&+\f{C\lan k, \xi-kt\ran^{\f12}\lan l, \eta-lt\ran^{\f12}}{\lan k-l, (\xi-\eta)-(k-l)t\ran^{\f12}\lan k-l, \xi-\eta\ran^{s}}
|M_0 (t, k, \xi)-M_0 (t, l, \eta)|\\
& +\f{C\lan k, \xi-kt\ran^{\f12}\lan l, \eta-lt\ran^{\f12}}{\lan k-l, (\xi-\eta)-(k-l)t\ran^{\f12}\lan k-l, \xi-\eta\ran^{s-1}\lan k, \xi\ran^{\f12}\lan l, \eta\ran^{\f12}} \\
\leq& \f{C\lan k-l, (\xi-\eta)-(k-l)t\ran^{\f12}}{\lan k-l, \xi-\eta\ran^{s}}\\
&+\f{C\lan k, \xi-kt\ran^{\f12}\lan l, \eta-lt\ran^{\f12}}{\lan k-l, (\xi-\eta)-(k-l)t\ran^{\f12}\lan k-l, \xi-\eta\ran^{s}}
\frac{\lan t\ran |k-l, \xi-\eta|}{\lan k, \xi-kt\ran+\lan l, \eta-lt\ran}\\
& +\f{C\lan t\ran^{\f12}\lan k, \xi\ran^{\f12}\lan t\ran^{\f12}\lan l, \eta\ran^{\f12}}{\lan k-l, (\xi-\eta)-(k-l)t\ran^{\f12}\lan k-l, \xi-\eta\ran^{s-1}\lan k, \xi\ran^{\f12}\lan l, \eta\ran^{\f12}}\\
\leq& \f{C\lan t\ran\lan k-l, \xi-\eta\ran^{\f12}}{\lan k-l, \xi-\eta\ran^{s}}+\f{C\lan t\ran |k-l, \xi-\eta|}{\lan k-l, (\xi-\eta)-(k-l)t\ran^{\f12}\lan k-l, \xi-\eta\ran^{s}}\\
& +\f{C\lan t\ran}{\lan k-l, (\xi-\eta)-(k-l)t\ran^{\f12}\lan k-l, \xi-\eta\ran^{s-1}}\\
\leq& \f{C\lan t\ran}{\lan k-l, \xi-\eta\ran^{s-\f12}}+\f{C\lan t\ran}{\lan k-l, (\xi-\eta)-(k-l)t\ran^{\f12}\lan k-l, \xi-\eta\ran^{s-1}}\\
\leq& C\lan t\ran(\lan k-l, \xi-\eta\ran^{\frac{1}{2}-s}+\lan k-l, (\xi-\eta)-(k-l)t\ran^{\frac{1}{2}-s}).
\end{align*}
Then  \eqref{Ak1} follows from 
$A_k(t, \xi)\sim\lan k, \xi-kt\ran^\frac{1}{2} \lan k, \xi\ran^s$, $A_l(t, \eta)\sim\lan l, \eta-lt\ran^{\f12}\lan l, \eta\ran^{s}$
 and\\ $ A_{k-l}(t, \xi-\eta)\sim\lan k-l, (\xi-\eta)-(k-l)t\ran^{\f12}\lan k-l, \xi-\eta\ran^{s}$.

 Now we consider  \eqref{Ak2}.

 \textbf{Case 1:}  
$\langle k-l,\xi-\eta\rangle\geq (\langle k,\xi\rangle+\langle l,\eta\rangle)/4$.
Since $A_k(t, \xi)\sim|k_+, \xi-kt|^\frac{1}{2} \lan k, \xi\ran^s$, we have
 \begin{align*}
&
\f{\lan l, \eta-lt\ran A_k(t, \xi)}{A_l(t, \eta)A_{k-l}(t, \xi-\eta)}\f{|k-l|}{|k_+, \xi-kt|}\\
\leq &\f{C\lan l, \eta-lt\ran|k_+, \xi-kt|^\frac{1}{2}|k-l|\lan k, \xi\ran^s}{|k_+, \xi-kt||l_+, \eta-lt|^\frac{1}{2}|(k-l)_+, (\xi-\eta)-(k-l)t|^\frac{1}{2}\lan l, \eta\ran^s\lan k-l, \xi-\eta\ran^s} \\
\leq&C \f{|k-l|^\frac{1}{2} }{|k_+, \xi-kt|^\frac{1}{2} } \f{\lan l, \eta-lt\ran^\f12}{\lan l, \eta\ran^s}\leq C \f{|k|^{\frac{1}{2}}+|l|^{\frac{1}{2}} }{\lan t\ran^{-\f12}\lan k, \xi\ran^{\f12} } \f{\lan t\ran^\f12 \lan l, \eta\ran^\f12}{\lan l, \eta\ran^s}= C\lan t\ran \f{|k|^{\frac{1}{2}}+|l|^{\frac{1}{2}}}{\lan k, \xi\ran^{\f12}\lan l, \eta\ran^{s-\f12}}\\
\leq&\f{C\lan t\ran}{\lan l, \eta\ran^{s-\f12}}+\f{C\lan t\ran}{\lan k, \xi\ran^{\f12}\lan l, \eta\ran^{s-1}}
\leq C\lan t\ran(\lan l, \eta\ran^{\f12-s}+\lan k, \xi\ran^{\f12-s}),
\end{align*}
where we use $\lan l, \eta-lt\ran^\f12\leq 2\lan t\ran^\f12 \lan l, \eta\ran^\f12$, $2|k_+, \xi-kt|^\frac{1}{2} \geq \lan t\ran^{-\f12}\lan k, \xi\ran^{\f12}$  and similarly 
\begin{align*}
&\f{\lan k, \xi-kt\ran A_l(t, \eta)}{A_k(t, \xi)A_{k-l}(t, \xi-\eta)}\f{|k-l|}{|l_+, \eta-lt|}
\leq C\lan t\ran(\lan l, \eta\ran^{\f12-s}+\lan k, \xi\ran^{\f12-s}),
\end{align*}
we  can get  
\begin{align*}
& |\lan l, \eta-lt\ran A_k^2(t, \xi)|\f{|k||k-l|}{k^2+(\xi-kt)^2}+|\lan k, \xi-kt\ran A_l^2(t, \xi)|\f{|l||k-l|}{l^2+(\eta-lt)^2}\\
\leq& \left|\f{\lan l, \eta-lt\ran A_k(t, \xi)}{A_l(t, \eta)A_{k-l}(t, \xi-\eta)}\f{|k-l|}{|k_+, \xi-kt|}+\f{\lan k, \xi-kt\ran A_l(t, \eta)}{A_k(t, \xi)A_{k-l}(t, \xi-\eta)}\f{|k-l|}{|l_+, \eta-lt|}
\right|\times\\&A_k(t, \xi) A_l(t, \eta)A_{k-l}(t, \xi-\eta) \\
\leq& C\lan t\ran A_k(t, \xi) A_l(t, \eta)A_{k-l}(t, \xi-\eta) (\lan l, \eta\ran^{\f12-s}+\lan k, \xi\ran^{\f12-s}).
\end{align*}

 \textbf{Case 2:}  
$\langle k-l,\xi-\eta\rangle\leq (\langle k,\xi\rangle+\langle l,\eta\rangle)/4$.
In this case $\langle k,\xi\rangle\sim\langle l,\eta\rangle $, and
 \begin{align*}
&\f{|k||k-l|}{k^2+(\xi-kt)^2}\f{ |\lan l, \eta-lt\ran A_k^2(t, \xi)|}{A_k(t, \xi) A_l(t, \eta)A_{k-l}(t, \xi-\eta)}\non\\
 \leq& \f{|k-l|}{ |k_+, \xi-kt|}\f{C\lan l, \eta-lt\ran  |k_+, \xi-kt| \lan k, \xi\ran^{2s}   }{\lan k, \xi-kt\ran^{\f12}\lan k, \xi\ran^{s}\lan l, \eta-lt\ran^{\f12}\lan l, \eta\ran^{s}\lan k-l, (\xi-\eta)-(k-l)t\ran^{\f12}\lan k-l, \xi-\eta\ran^{s}}\\
  \leq& \f{C|k-l|\lan l, \eta-lt\ran^\f12}{\lan k, \xi-kt\ran^{\f12}\lan k-l, (\xi-\eta)-(k-l)t\ran^{\f12}\lan k-l, \xi-\eta\ran^{s}}\\
    \leq& \f{C}{\lan k-l, \xi-\eta\ran^{s-1}}\Big(\f{1}{\lan k-l, (\xi-\eta)-(k-l)t\ran^{1/2}}+\f{1}{\lan k, \xi-kt\ran^{1/2}}\Big)\\
    \leq& \f{C\lan t\ran}{\lan k-l, \xi-\eta\ran^{s-1}}\Big(\f{1}{\lan k-l, \xi-\eta\ran^{1/2}}+\f{1}{\lan k, \xi\ran^{1/2}}\Big)\\
  \leq& C\lan t\ran (\lan k, \xi\ran^{\f12-s}+\lan k-l, \xi-\eta\ran^{\f12-s}).
\end{align*}
Here we used
 $A_k(t, \xi)\sim|k_+, \xi-kt|^\frac{1}{2} \lan k, \xi\ran^s\sim\lan k, \xi-kt\ran^\frac{1}{2} \lan k, \xi\ran^s,$ 
 (and similar bounds for $A_l(t, \eta)$, $A_{k-l}(t, \xi-\eta)$),
 and
 \begin{align*}
\lan l, \eta-lt\ran
&\leq\lan k, \xi-kt\ran+\lan k-l, (\xi-\eta)-(k-l)t\ran\\
&=  \lan k, \xi-kt\ran \lan k-l, (\xi-\eta)-(k-l)t\ran\Big(\f{1}{\lan k-l, (\xi-\eta)-(k-l)t\ran}+\f{1}{\lan k, \xi-kt\ran}\Big).
\end{align*}
 Then we obtain \eqref{Ak2}.
  \end{proof}
  By Young's convolution inequality, H\"older's inequality (and $\lan k, \xi\ran^{\frac{1}{2}-s}\in L_{k,\xi}^2$) we have 
\begin{lemma}\label{lemplus}
For $s>3/2$, it holds that
\begin{align*}
&\sum_{k, l\in \mathbb{Z}}\int_{\mathbb{R}^2}
 |\hat{f}(k-l, \xi-\eta )||\hat{g}(l, \eta)||{\hat{h}}(k, \xi)|(\lan k, \xi\ran^{\frac{1}{2}-s}+\lan l, \eta\ran^{\frac{1}{2}-s}+\lan k-l, \xi-\eta\ran^{\frac{1}{2}-s}\\&+\lan k-l, (\xi-\eta)-(k-l)t\ran^{\frac{1}{2}-s})d\xi d\eta
 \leq C\|\hat{f}\|_{L_{k,\xi}^2}\|\hat{g}\|_{L_{k,\xi}^2}\|\hat{h}\|_{L_{k,\xi}^2}.
\end{align*}\end{lemma}

\smallskip

\section{Energy estimate in small time scale $t<\kappa^{-\frac{1}{6}}$}

\setcounter{equation}{0}

As the domain is $\mathbb{T}\times\mathbb{R}$, the space-time estimates can be established through the use of  the Fourier
transform in $(X,Y)$. Taking  Fourier transform  to \eqref{eq: u new}, we have
\begin{equation}\label{eq: ub1}
 \left\{\begin{array}{l}
\hat{u}_t+\nu(k^2+(\xi-kt)^2) \hat{u}-\hat{u}_2\big(\f{k^2-(\xi-kt)^2}{k^2+(\xi-kt)^2},\f{2k(\xi-kt)}{k^2+(\xi-kt)^2}\big)\\
+\g^2\hat{\vth}\big(-\f{k(\xi-kt)}{k^2+(\xi-kt)^2},\f{k^2}{k^2+(\xi-kt)^2}\big)=\hat{F},\\
\hat{\vth}_t+\mu(k^2+(\xi-kt)^2) \hat{\vth}-\hat{u}_2=\hat{G},\\
k \hat{u}_1+(\xi-kt) \hat{u}_2=0.
\end{array}\right.
\end{equation}
Here we used  $\hat{p}_L=\f{2ik\hat{u}_2+i\g^2 (\xi-kt) \hat{\vth}}{k^2+(\xi-kt)^2}$ and
\begin{align*}
&(\hat{u}_2,0)+i(k,\xi-kt)\hat{p}_L +(0, \g^2\hat{ \vth})\\
=&-\hat{u}_2\bigg(\f{k^2-(\xi-kt)^2}{k^2+(\xi-kt)^2},\f{2k(\xi-kt)}{k^2+(\xi-kt)^2}\bigg)+
\g^2\hat{\vth}\bigg(-\f{k(\xi-kt)}{k^2+(\xi-kt)^2},\f{k^2}{k^2+(\xi-kt)^2}\bigg).
\end{align*}

\begin{proposition}\label{pr1}
Let $(u,\vth)$ be a solution of \eqref{eq: ub1}. If $\kappa\in(0, 1)$, $\g>\f12$, $s>\f32$ and $\f{\nu+\mu}{2\g\sqrt{\nu \mu} }< 2-\e$, $0<\e<1/\g$, then
\begin{align}\label{t small}
&\f{d}{dt}\sum_{k\in \mathbb{Z}}\int_{\mathbb{R}} \Big(\big|A_k\hat{u}\big|^2+\g^2\big|A_k\hat{\vth}\big|^2+\Re (A_k\hat{\vth}A_k \overline{\hat{u}}_1)\Big) d\xi +\sum_{k\in \mathbb{Z}}\int_{\mathbb{R}} \f{\g  k^2}{k^2+(\xi-kt)^2}\big|A_k\hat{\vth}\big|^2 d\xi\nonumber\\
&+\f{2}{\g} \sum_{k\in \mathbb{Z}}\int_{\mathbb{R}}  |A_k \hat{u}_2|^2 d\xi+\e \sum_{k\in \mathbb{Z}}\int_{\mathbb{R}} (k^2+(\xi-kt)^2)(\nu\big|A_k\hat{u}\big|^2+\mu\gamma^2\big|A_k\hat{\vth}\big|^2) d\xi\\ \nonumber
 \leq&C\lan t\ran\|A_k(\hat{u},\hat{\vth})\|_{{L_{k,\xi}^2}}^3.
\end{align}
\end{proposition}
 \begin{proof}
 
The system for the weighted variables  $(A_k\hat{u}, A_k\hat{\vth})$ read as
\begin{equation}\label{ub2}
 \left\{\begin{array}{l}
(A_k\hat{u})_t-q A_k\hat{u}+\nu(k^2+(\xi-kt)^2) A_k\hat{u}\\
\quad-\big(\f{k^2-(\xi-kt)^2}{k^2+(\xi-kt)^2}, \f{2k(\xi-kt)}{k^2+(\xi-kt)^2}\big)A_k\hat{u}_2+\g^2\big(-\f{k(\xi-kt)}{k^2+(\xi-kt)^2}, \f{k^2}{k^2+(\xi-kt)^2}\big) A_k\hat{\vth} =A_k \hat{F},\\
(A_k\hat{\vth})_t-qA_k\hat{\vth}+\mu(k^2+(\xi-kt)^2) A_k\hat{\vth}-A_k\hat{u}_2=A_k \hat{G},\\
k A_k \hat{u}_1+(\xi-kt)  A_k \hat{u}_2=0,
\end{array}\right.
\end{equation}
here we define $q:=\f{\p_t A_k}{A_k}$, then
\begin{align}\label{def:q}
q:=-\f{ k (\xi-kt)}{2(k^2+(\xi-kt)^2)}-\f{  C_\g k^2}{k^2+(\xi-kt)^2},\end{align}
where $ C_\g=\max\{1,\f{2}{2\g-1}\} $ is determined according to the proof. 

Multiplying   $ \eqref{ub2}_1$ with $2 A_k\overline{\hat{u}}$,  and  then integrating   with respect to $\xi$, subsequently summing over $k$, taking the real part, we get
\begin{align}
& \f{d}{dt}\sum_{k\in \mathbb{Z}}\int_{\mathbb{R}} |A_k \hat{u}|^2d\xi-2\sum_{k\in \mathbb{Z}}\int_{\mathbb{R}} q\big|A_k\hat{u}\big|^2d\xi+2\nu\sum_{k\in \mathbb{Z}}\int_{\mathbb{R}}  (k^2+(\xi-kt)^2)\big|A_k\hat{u}\big|^2 d\xi \nonumber\\
&-2\Re\sum_{k\in \mathbb{Z}}\int_{\mathbb{R}}\f{k^2-(\xi-kt)^2}{k^2+(\xi-kt)^2}A_k \hat{u}_2A_k\overline{\hat{u}}_1d\xi +2\g^2\Re\sum_{k\in \mathbb{Z}}\int_{\mathbb{R}}\f{(\xi-kt)^2}{k^2+(\xi-kt)^2}A_k\hat{\vth} A_k\overline{\hat{u}}_2d\xi\nonumber\\
&+2\Re\sum_{k\in \mathbb{Z}}\int_{\mathbb{R}}\f{2k^2}{k^2+(\xi-kt)^2}A_k \hat{u}_2A_k\overline{\hat{u}}_1d\xi +2\g^2\Re\sum_{k\in \mathbb{Z}}\int_{\mathbb{R}}\f{k^2}{k^2+(\xi-kt)^2}A_k\hat{\vth} A_k\overline{\hat{u}}_2d\xi\nonumber\\
=&2\Re\sum_{k\in \mathbb{Z}}\int_{\mathbb{R}} A_k \hat{F}  A_k\overline{\hat{u}},\label{u1}
\end{align}
here we used $kA_k \hat{u}_1+(\xi-kt) A_k\hat{u}_2=0$ (see $\eqref{ub2}_3$) and 
\begin{align*}
-\f{k(\xi-kt)}{(k^2+(\xi-kt)^2)^\f12}A_k\hat{\vth} A_k\overline{\hat{u}}_1=\f{(\xi-kt)^2}{(k^2+(\xi-kt)^2)^\f12}A_k\hat{\vth} A_k\overline{\hat{u}}_2,\\
-\f{2k(\xi-kt)}{(k^2+(\xi-kt)^2)^\f12}A_k \hat{u}_2A_k\overline{\hat{u}}_2=\f{2k^2}{(k^2+(\xi-kt)^2)^\f12}A_k \hat{u}_2A_k\overline{\hat{u}}_1.
\end{align*}

Multiplying   $ \eqref{ub2}_2$ with $2 A_k\overline{\hat{\vth}}$,  and  then integrating   with respect to $\xi$, subsequently summing over $k$, taking the real part, we get
\begin{align}
& \frac{d}{dt}\sum_{k\in \mathbb{Z}}\int_{\mathbb{R}}\big| A_k\hat{\vth}\big|^2d\xi-2\sum_{k\in \mathbb{Z}}\int_{\mathbb{R}} q\big|A_k\hat{\vth}\big|^2d\xi+2\mu \sum_{k\in \mathbb{Z}}\int_{\mathbb{R}} (k^2+(\xi-kt)^2)\big|A_k\hat{\vth}\big|^2d\xi\nonumber\\
&-2\Re\sum_{k\in \mathbb{Z}}\int_{\mathbb{R}}A_k \hat{u}_2 A_k\overline{\hat{\vth}}d\xi 
=2\Re\sum_{k\in \mathbb{Z}}\int_{\mathbb{R}}A_k \hat{G} A_k\overline{\hat{\vth}}d\xi.\label{vth}
\end{align}
Combining  \eqref{u1} and \eqref{vth} together, we have
\begin{align}\label{estimate-3}
& \frac{d}{dt}\sum_{k\in \mathbb{Z}}\int_{\mathbb{R}}\big| A_k\hat{u}\big|^2+\gamma^2\big| A_k\hat{\vth}\big|^2d\xi-2\sum_{k\in \mathbb{Z}}\int_{\mathbb{R}} q(\big|A_k\hat{u}\big|^2+\gamma^2\big|A_k\hat{\vth}\big|^2)d\xi\nonumber\\
&+2 \sum_{k\in \mathbb{Z}}\int_{\mathbb{R}} (k^2+(\xi-kt)^2)(\nu\big|A_k\hat{u}\big|^2+\mu\gamma^2\big|A_k\hat{\vth}\big|^2) d\xi+2 \Re\sum_{k\in \mathbb{Z}}\int_{\mathbb{R}}A_k\hat{u}_2A_k\overline{\hat{u}}_1 d\xi \nonumber\\
=&2\Re\sum_{k\in \mathbb{Z}}\int_{\mathbb{R}}A_k \hat{F} A_k\overline{\hat{u}}d\xi+2\g^2\Re\sum_{k\in \mathbb{Z}}\int_{\mathbb{R}}  A_k \hat{G} A_k\overline{\hat{\vth}} d\xi.
\end{align}
Additionally, we mention that {(using \eqref{def:q} and $ \eqref{ub2}_3$)}
\begin{align}\label{linear L}
-2\sum_{k\in \mathbb{Z}}\int_{\mathbb{R}} q\big|A_k\hat{u}\big|^2 d\xi
&=2\sum_{k\in \mathbb{Z}}\int_{\mathbb{R}} \Big(\f{ k (\xi-kt)}{2(k^2+(\xi-kt)^2)}+\f{  C_\g k^2}{k^2+(\xi-kt)^2}\Big)\big|A_k\hat{u}\big|^2 d\xi\nonumber\\
&=-\Re\sum_{k\in \mathbb{Z}}\int_{\mathbb{R}}A_k\hat{u}_2A_k\overline{\hat{u}}_1 d\xi+2C_\g \sum_{k\in \mathbb{Z}}\int_{\mathbb{R}}  |A_k \hat{u}_2|^2 d\xi,
\end{align}
where we use
\begin{align*}
\f{k (\xi-kt)}{k^2+(\xi-kt)^2 }| \hat{u}|^2&=\f{k (\xi-kt)}{k^2+(\xi-kt)^2 }( \hat{u}_1\overline{ \hat{u}}_1+ \hat{u}_2\overline{ \hat{u}}_2)\nonumber\\
&=-\f{(\xi-kt)^2}{k^2+(\xi-kt)^2 }\Re ( \hat{u}_2 \overline{ \hat{u}}_1)-\f{k^2}{k^2+(\xi-kt)^2 }\Re ( \hat{u}_1\overline{ \hat{u}}_2)=-\Re ( \hat{u}_2 \overline{ \hat{u}}_1),
\end{align*}
and
\begin{align}\label{u2}
\big|\hat{u}\big|^2=\big|\hat{u}_1\big|^2+\big|\hat{u}_2\big|^2=\f{|\xi-kt|^2}{k^2}\big|\hat{u}_2\big|^2+\big|\hat{u}_2\big|^2=\f{k^2+(\xi-kt)^2 } {k^2}\big|\hat{u}_2\big|^2.
\end{align}
So \eqref{estimate-3} can be rewritten as 
\begin{align}\label{estimate-4}
& \frac{d}{dt}\sum_{k\in \mathbb{Z}}\int_{\mathbb{R}}\big| A_k\hat{u}\big|^2+\gamma^2\big| A_k\hat{\vth}\big|^2d\xi+2C_\g \sum_{k\in \mathbb{Z}}\int_{\mathbb{R}}  |A_k \hat{u}_2|^2 d\xi+2C_\g\sum_{k\in \mathbb{Z}}\int_{\mathbb{R}} \f{ \g^2 k^2}{k^2+(\xi-kt)^2}\big|A_k\hat{\vth}\big|^2 d\xi\nonumber\\
&+2 \sum_{k\in \mathbb{Z}}\int_{\mathbb{R}} (k^2+(\xi-kt)^2)(\nu\big|A_k\hat{u}\big|^2+\mu\gamma^2\big|A_k\hat{\vth}\big|^2) d\xi+\Re\sum_{k\in \mathbb{Z}}\int_{\mathbb{R}}A_k\hat{u}_2A_k\overline{\hat{u}}_1 d\xi \nonumber\\
&+\Re\sum_{k\in \mathbb{Z}}\int_{\mathbb{R}} \f{ \g^2 k (\xi-kt)}{k^2+(\xi-kt)^2}\big|A_k\hat{\vth}\big|^2 d\xi
=2\Re\sum_{k\in \mathbb{Z}}\int_{\mathbb{R}}A_k \hat{F} A_k\overline{\hat{u}}d\xi+2\g^2\Re\sum_{k\in \mathbb{Z}}\int_{\mathbb{R}}  A_k \hat{G} A_k\overline{\hat{\vth}} d\xi.
\end{align}

We calculate the time derivative of  $A_k\hat{\vth}A_k \overline{\hat{u}}_1$ 
\begin{align}\label{c1}
&\f{d}{dt}\Big( A_k\hat{\vth}A_k \overline{\hat{u}}_1\Big)\nonumber\\
=& (A_k\hat{\vth})_tA_k \overline{\hat{u}}_1+A_k\hat{\vth}(A_k \overline{\hat{u}}_{1})_t\nonumber\\
=& \Big(q A_k\hat{\vth}-\mu(k^2+(\xi-kt)^2)) A_k\hat{\vth}+  A_k\hat{u}_2+A_k \hat{G}\Big) A_k \overline{\hat{u}}_1\nonumber\\
&+A_k\hat{\vth}\Big(q A_k \overline{\hat{u}}_1-\nu(k^2+(\xi-kt)^2) A_k \overline{\hat{u}}_1\nonumber\\&+\f{k^2-(\xi-kt)^2}{k^2+(\xi-kt)^2}A_k \overline{\hat{u}}_2+\g^2 \f{k (\xi-kt)}{k^2+(\xi-kt)^2} A_k \overline{\hat{\vth}}+A_k \overline{\hat{F}_1}\Big)\nonumber\\
=&2q A_k\hat{\vth}A_k \overline{\hat{u}}_1-(\nu+\mu) (k^2+(\xi-kt)^2) A_k\hat{\vth}A_k \overline{\hat{u}}_1+A_k \hat{G} A_k \overline{\hat{u}}_1\nonumber\\
&+\f{k^2-(\xi-kt)^2}{k^2+(\xi-kt)^2} A_k\hat{\vth}A_k \overline{\hat{u}}_2+\f{\g^2k(\xi-kt)}{k^2+(\xi-kt)^2} |A_k\hat{\vth}|^2
+  A_k\hat{u}_2A_k \overline{\hat{u}}_1+A_k\hat{\vth}A_k \overline{\hat{F}}_1.
\end{align}

 Integrating   with respect to $\xi$, subsequently summing over $k$, taking the real part of \eqref{c1}, and adding it to  \eqref{estimate-4}, we have
  \begin{align}\label{rs2}
& \f{d}{dt}\sum_{k\in \mathbb{Z}}\int_{\mathbb{R}} \Big(\big|A_k\hat{u}\big|^2+\g^2\big|A_k\hat{\vth}\big|^2+\Re (A_k\hat{\vth}A_k \overline{\hat{u}}_1)\Big) d\xi\nonumber\\
& +2C_\g \sum_{k\in \mathbb{Z}}\int_{\mathbb{R}}  |A_k \hat{u}_2|^2 d\xi+2C_\g\sum_{k\in \mathbb{Z}}\int_{\mathbb{R}} \f{ \g^2 k^2}{k^2+(\xi-kt)^2}\big|A_k\hat{\vth}\big|^2 d\xi\nonumber\\
&+2 \sum_{k\in \mathbb{Z}}\int_{\mathbb{R}} (k^2+(\xi-kt)^2)(\nu\big|A_k\hat{u}\big|^2+\mu\gamma^2\big|A_k\hat{\vth}\big|^2) d\xi \nonumber\\
& =\Re\sum_{k\in \mathbb{Z}}\int_{\mathbb{R}}\Big(-\underbrace{(\nu+\mu)(k^2+(\xi-kt)^2)( A_k\hat{\vth}A_k \overline{\hat{u}}_1) }_{\mathcal{I}_1}\nonumber\\&+\underbrace{2q (A_k\hat{\vth}A_k \overline{\hat{u}}_1) +\f{k^2-(\xi-k t)^2}{k^2+(\xi-kt)^2} (A_k\hat{\vth}A_k \overline{\hat{u}}_2)}_{\mathcal{I}_2}\nonumber\\
 &+ \underbrace{A_k \hat{G} A_k \overline{\hat{u}}_1+A_k\hat{\vth}A_k\overline{ \hat{F}}_1}_{\mathcal{I}_3}+\underbrace{2\g^2 A_k \hat{ G}  A_k \overline{\hat{\vth}}}_{\mathcal{I}_4}+ \underbrace{2A_k \hat{F} A_k \overline{\hat{u}}}_{\mathcal{I}_5}\Big) d\xi.
\end{align}
What's more the  linear term  on the right hand side  of  \eqref{rs2} can be estimated as
\begin{align}\label{I1}
{\left|\mathcal{I}_1\right|}\leq&(\nu+\mu)\big(k^2+(\xi-kt)^2\big)\big|A_k\hat{\vth}\big|\big|A_k\hat{u}_1\big|\nonumber\\
\leq&\f{\nu+\mu}{2\g\sqrt{\nu \mu} }(k^2+(\xi-kt)^2)(\nu\big|A_k\hat{u}\big|^2+\mu\g^2\big|A_k\hat{\vth}\big|^2),
\end{align}
and  by the divergence free condition $ \eqref{ub2}_3$,
\begin{align}\label{I2}
|\mathcal{I}_2|
=&\left|-2\Big(\f{ k (\xi-kt)}{2(k^2+(\xi-kt)^2)}+\f{  C_\g k^2}{k^2+(\xi-kt)^2}\Big)(A_k\hat{\vth}A_k \overline{\hat{u}}_1)+\f{k^2-(\xi-k t)^2}{k^2+(\xi-kt)^2} (A_k\hat{\vth}A_k \overline{\hat{u}}_2)\right|\nonumber\\
=&\left|-\f{2  C_\g k^2}{k^2+(\xi-kt)^2}(A_k\hat{\vth}A_k \overline{\hat{u}}_1)+\f{k^2}{k^2+(\xi-kt)^2} (A_k\hat{\vth}A_k \overline{\hat{u}}_2)\right|\nonumber\\
\leq& \f{C_{\g} k^2}{\g(k^2+(\xi-kt)^2)} (\g^2|A_k\hat{\vth}|^2+\f{|\xi-kt|^2}{k^2}|A_k\hat{u}_2|^2)\nonumber\\&+ \f{ k^2}{\g(k^2+(\xi-kt)^2)} (\g^2|A_k\hat{\vth}|^2+|A_k\hat{u}_2|^2)\nonumber\\
\leq &\f{C_{\g}+1}{\g}\f{ k^2}{k^2+(\xi-kt)^2} \g^2|A_k\hat{\vth}|^2+\f{C_{\g}}{\g}|A_k\hat{u}_2|^2,
\end{align}
here we used $C_\g\geq 1 $. 

Summing up \eqref{I1}, \eqref{I2} into  \eqref{rs2},  $\f{\nu+\mu}{2\g\sqrt{\nu \mu} }< 2-\e$ and $\f{C_{\g}+2}{\g}\leq 2C_\g$ (as $\f{2}{2\g-1}\leq C_\g$), we have
\begin{align}\label{rs3}
&\f{d}{dt}\sum_{k\in \mathbb{Z}}\int_{\mathbb{R}} \Big(\big|A_k\hat{u}\big|^2+\g^2\big|A_k\hat{\vth}\big|^2+\Re (A_k\hat{\vth}A_k \overline{\hat{u}}_1)\Big) d\xi +\sum_{k\in \mathbb{Z}}\int_{\mathbb{R}} \f{\g  k^2}{k^2+(\xi-kt)^2}\big|A_k\hat{\vth}\big|^2 d\xi\nonumber\\
&+\f{2}{\g} \sum_{k\in \mathbb{Z}}\int_{\mathbb{R}}  |A_k \hat{u}_2|^2 d\xi+\e \sum_{k\in \mathbb{Z}}\int_{\mathbb{R}} (k^2+(\xi-kt)^2)(\nu\big|A_k\hat{u}\big|^2+\mu\gamma^2\big|A_k\hat{\vth}\big|^2) d\xi \nonumber\\
 &\leq\Re\sum_{k\in \mathbb{Z}}\int_{\mathbb{R}}\Big( \underbrace{A_k \hat{G} A_k \overline{\hat{u}}_1+A_k\hat{\vth}A_k\overline{ \hat{F}}_1}_{\mathcal{I}_3}+\underbrace{2\g^2A_k \hat{ G}  A_k \overline{\hat{\vth}}}_{\mathcal{I}_4} + \underbrace{2A_k \hat{F} A_k \overline{\hat{u}}}_{\mathcal{I}_6}\Big) d\xi.
\end{align}
All the right side of \eqref{rs3} are nonlinear terms. For $I_4$, we use symmetrization to have
\begin{align}\label{I4}
&\Big|\Re\sum_{k\in \mathbb{Z}}\int_{\mathbb{R}}\mathcal{I}_4 d\xi\Big|=\Big|2\g^2\Re\sum_{k\in \mathbb{Z}}\int_{\mathbb{R}}A_k \hat{G} A_k \overline{\hat{\vth}} d\xi\Big|\nonumber\\
=&\Big|2\g^2\Re\sum_{k, l\in \mathbb{Z}}\int_{\mathbb{R}^2}( \hat{u}_1(k-l, \xi-\eta )il\hat{\vth}(l, \eta) +  \hat{u}_2(k-l, \xi-\eta) i (\eta-l t)\hat{\vth}(l, \eta) )A_k^2\overline{\hat{\vth}}(k, \xi) d\xi d\eta\Big| \nonumber\\
\leq &C\sum_{k, l\in \mathbb{Z}}\int_{\mathbb{R}^2}\Big|l A_k^2 (t, \xi)-kA_l^2 (t, \eta) \Big||\hat{u}_1(k-l, \xi-\eta )||\hat{\vth}(l, \eta)||\overline{\hat{\vth}}(k, \xi) |\nonumber\\
&+\Big|(\eta-lt) A_k^2 (t, \xi)-(\xi-kt) A_l^2 (t, \eta) \Big| |\hat{u}_2(k-l, \xi-\eta )||\hat{\vth}(l, \eta)||\overline{\hat{\vth}}(k, \xi)| d\xi d\eta \nonumber\\
\leq&C \sum_{k, l\in \mathbb{Z}}\int_{\mathbb{R}^2}|\hat{u}(k-l, \xi-\eta )||\hat{\vth}(l, \eta)||\overline{\hat{\vth}}(k, \xi)| \lan t\ran A_k(t, \xi) A_l(t, \eta)A_{k-l}(t, \xi-\eta)\times\nonumber\\
&\quad  (\lan k, \xi\ran^{\frac{1}{2}-s}+\lan l, \eta\ran^{\frac{1}{2}-s}+\lan k-l, \xi-\eta\ran^{\frac{1}{2}-s}+\lan k-l, \xi-\eta-(k-l)t\ran^{\frac{1}{2}-s})  d\xi d\eta \nonumber\\
\leq &C\lan t\ran\|A_k  \hat{\vth}\|_{L^2_{k, \xi}}\|A_k  \hat{\vth}\|_{L^2_{k, \xi}}\|A_k  \hat{u}\|_{L^2_{k, \xi}},
\end{align}
here we used that $\hat{u}(k-l, \xi-\eta)=\overline{\hat{u}}(l-k,\eta-\xi)$, Lemma \ref{Ak} and  Lemma \ref{lemplus}.

For $\mathcal{I}_3$, we use $\hat{u}(k-l, \xi-\eta)=\overline{\hat{u}}(l-k,\eta-\xi)$, $-\Delta_Lp_{NL}=
2((\na_L^1u_1)^2+\na_L^1u_2\na_L^2u_1)$, Lemma \ref{Ak} and  Lemma \ref{lemplus}, then
\begin{align}\label{I5}
&\Big|\Re\sum_{k\in \mathbb{Z}}\int_{\mathbb{R}}\mathcal{I}_3d\xi\Big|=\Big|\Re\sum_{k\in \mathbb{Z}}\int_{\mathbb{R}}A_k \hat{G} A_k \overline{\hat{u}}_1+A_k\hat{\vth}A_k\overline{ \hat{F}}_1d\xi\Big|\nonumber\\
=&\Big|\Re\sum_{k, l\in \mathbb{Z}}\int_{\mathbb{R}^2}\Big[( \hat{u}_1(k-l, \xi-\eta )il\hat{\vth}(l, \eta) +  \hat{u}_2(k-l, \xi-\eta) i (\eta-l t)\hat{\vth}(l, \eta) )A_k^2\overline{\hat{u}}_1(k, \xi)\non\\
&-A_l^2\hat{\vth}(l, \eta)\big( \hat{u}_1(k-l, \xi-\eta )ik\overline{\hat{u}}_1(k, \xi) +  \hat{u}_2(k-l, \xi-\eta) i (\xi-kt)\overline{\hat{u}}_1(k, \xi) \nonumber\\
&+\f{2il(k-l)}{l^2+(\eta-lt)^2}(\hat{u}_1(k-l, \xi-\eta )k\overline{\hat{u}}_1(k, \xi)+
\hat{u}_2(k-l, \xi-\eta )(\xi-kt)\overline{\hat{u}}_1(k, \xi)  \big) \Big]d\xi d\eta\Big| \nonumber\\
\leq &\sum_{k, l\in \mathbb{Z}}\int_{\mathbb{R}^2}\Big[\Big|l A_k^2 (t, \xi)-kA_l^2 (t, \eta) \Big||\hat{u}_1(k-l, \xi-\eta )||\hat{\vth}(l, \eta)||\overline{\hat{u}}_1(k, \xi) | \nonumber\\
&+\Big|(\eta-lt) A_k^2 (t, \xi)-(\xi-kt) A_l^2 (t, \eta) \Big| |\hat{u}_2(k-l, \xi-\eta )||\hat{\vth}(l, \eta)||\overline{\hat{u}}_1(k, \xi)|\nonumber\\
&+A_l^2 (t, \eta)\f{2|l||k-l||k,\xi-kt|}{l^2+(\eta-lt)^2} |\hat{u}(k-l, \xi-\eta )||\hat{\vth}(l, \eta)||\overline{\hat{u}}_1(k, \xi)|\Big]d\xi d\eta \nonumber\\
\leq& C\sum_{k, l\in \mathbb{Z}}\int_{\mathbb{R}^2}|\hat{u}(k-l, \xi-\eta )||\hat{\vth}(l, \eta)||\overline{\hat{u}}(k, \xi)|\lan t\ran A_k(t, \xi) A_l(t, \eta)A_{k-l}(t, \xi-\eta)\nonumber\\
&\quad  (\lan k, \xi\ran^{\frac{1}{2}-s}+\lan l, \eta\ran^{\frac{1}{2}-s}+\lan k-l, \xi-\eta\ran^{\frac{1}{2}-s}+\lan k-l, (\xi-\eta)-(k-l)t\ran^{\frac{1}{2}-s})  d\xi d\eta \nonumber\\
\leq &\lan t\ran\|A_k \hat{u}_1\|_{L^2_{k, \xi}}\|A_k \hat{ \vth}\|_{L^2_{k, \xi}}\|A_{k}  \hat{u}_1\|_{L^2_{k, \xi}}.
\end{align}

For $I_5$, we use $\hat{u}(k-l, \xi-\eta)=\overline{\hat{u}}(l-k,\eta-\xi)$, $ik A_k\hat{u}_1+i(\xi-kt) A_k\hat{u}_2=0$ in $ \eqref{ub2}_3$, Lemma \ref{Ak} and  Lemma \ref{lemplus}, then
\begin{align}\label{I6}
&\Big|\Re\sum_{k\in \mathbb{Z}}\int_{\mathbb{R}}\mathcal{I}_5 d\xi\Big|=\Big|2\Re\sum_{k\in \mathbb{Z}}\int_{\mathbb{R}}A_k \hat{F} A_k \overline{\hat{u}}d\xi\Big|\nonumber\\
=&\Big|2\Re\sum_{k, l\in \mathbb{Z}}\int_{\mathbb{R}^2}( \hat{u}_1(k-l, \xi-\eta )il\hat{u}(l, \eta) +  \hat{u}_2(k-l, \xi-\eta) i (\eta-l t)\hat{u}(l, \eta) ) A_k^2\overline{\hat{u}}(k, \xi)d\xi d\eta\non\\
&+2\Re\sum_{k\in \mathbb{Z}}\int_{\mathbb{R}}(ik\hat{p}_{NL} A_k^2\overline{\hat{u}}_1(k, \xi)+
i(\xi-kt)\hat{p}_{NL} A_k^2\overline{\hat{u}}_2(k, \xi))
d\xi\Big| \nonumber\\
\leq &\sum_{k, l\in \mathbb{Z}}\int_{\mathbb{R}^2}\Big[\Big|l A_k^2 (t, \xi)-kA_l^2 (t, \eta) \Big||\hat{u}_1(k-l, \xi-\eta )||\hat{u}(l, \eta)||\overline{\hat{u}}(k, \xi) | \nonumber\\
&+\Big|(\eta-lt) A_k^2 (t, \xi)-(\xi-kt) A_l^2 (t, \eta) \Big| |\hat{u}_2(k-l, \xi-\eta )||\hat{u}(l, \eta)||\overline{\hat{u}}(k, \xi)|\nonumber\\
\leq&C \sum_{k, l\in \mathbb{Z}}\int_{\mathbb{R}^2}|\hat{u}(k-l, \xi-\eta )||\hat{u}(l, \eta)||\overline{\hat{u}}(k, \xi)|\lan t\ran A_k(t, \xi) A_l(t, \eta)A_{k-l}(t, \xi-\eta)\nonumber\\
&\quad  (\lan k, \xi\ran^{\frac{1}{2}-s}+\lan l, \eta\ran^{\frac{1}{2}-s}+\lan k-l, \xi-\eta\ran^{\frac{1}{2}-s}+\lan k-l, (\xi-\eta)-(k-l)t\ran^{\frac{1}{2}-s})  d\xi d\eta \nonumber\\
\leq &\lan t\ran\|A_k  \hat{u}\|_{L^2_{k, \xi}}^3.
\end{align}
Summing up \eqref{I4}, \eqref{I5}, \eqref{I6} into   \eqref{rs3} we obtain Proposition \ref{pr1}.
\end{proof}
\begin{proof}[Proof of Proposition \ref{pr01}]
By Proposition \ref{pr1}, we have
\begin{align*}
\f{d}{dt} E(t)\leq C\lan t\ran\|A_k(\hat{u},\hat{\vth})\|_{{L_{k,\xi}^2}}^3,\quad E(t):=\sum_{k\in \mathbb{Z}}\int_{\mathbb{R}} \Big(\big|A_k\hat{u}\big|^2+\g^2\big|A_k\hat{\vth}\big|^2+\Re (A_k\hat{\vth}A_k \overline{\hat{u}}_1)\Big) d\xi.
\end{align*}As $\g>\f12$, $A_k(t, \xi)\sim\lan k, \xi-kt\ran^\frac{1}{2} \lan k, \xi\ran^s $, we have $E(t) \sim \|A_k(\hat{u},\hat{\vth})\|_{{L_{k,\xi}^2}}^2\sim\| \lan\na_L\ran^\f12 (u, \vth)\|_{H^s }^2$. 
Thus $\f{d}{dt} E(t)\leq C\lan t\ran E(t)^{3/2}$, and by Gr\"onwall’s inequality, we have 
$$
E(t)^{1/2}\leq \frac{E(0)^{1/2}}{1-C (t+t^2)E(0)^{1/2}}\leq \f{C\| (u_{in}, \vth_{in})\|_{H^{s+1/2}}}{1-C_{+}(t+t^2)\ran \| (u_{in}, \vth_{in})\|_{H^{s+1/2}}}\leq \f{C\| (u_{in}, \vth_{in})\|_{H^{s+1/2}}}{1-C_{+}(t+t^2)\ran c_1\kappa^\f13},
$$ 
as long as $1-C_{+}(t+t^2)\ran c_1\kappa^\f13>0$, where $C_+>0$ depends only on $\g,s,\e$, and we used $E(0)^{1/2}\sim \| \lan\na_L\ran^\f12 (u, \vth)\|_{H^s }|_{t=0}=\| (u_{in}, \vth_{in})\|_{H^{s+1/2}}\leq c_1\kappa^\f13$.

Now we take $c_1=\f{1}{4C_+}$. Then for $ 0\leq t \leq T_0= \kappa^{-\f16}$, we have
$$C_{+}(t+ t^2 ) c_1\kappa^\f13\leq C_{+}(\kappa^{-\f16}+\kappa^{-\f13})c_1\kappa^{\f13}\leq 2C_{+}c_1 =1/2,$$
and then $\| \lan\na_L\ran^\f12 (u, \vth)(t)\|_{H^s }\le CE(t)^{1/2}\le  { C}\| (u_{in}, \vth_{in})\|_{H^{s+1/2}}$ 
 for $t\leq T_0=\kappa^{-\f16}$.
\end{proof}

\smallskip
\setcounter{equation}{0}

\section{ Property of the Multiplier $\mathcal{M}$ in long time scale $t>\kappa^{-\frac{1}{6}}$}
Recall that $\mathcal{M}(t, k, \xi)= |k_+, \xi-kt|^\frac{1}{2} \lan k, \xi\ran^s\mathcal{A} e^{\mathcal{M}_0}$, where $\mathcal{M}_0=\mathcal{M}_1+\mathcal{M}_2+\mathcal{M}_3$
with $\mathcal{M}_1(t, 0, \xi)=\mathcal{M}_2(t, 0, \xi)=0$ and
\begin{align*}
&\mathcal{M}_1(t, k, \xi)={\psi_{1}}\big(\kappa^{\f13}|k|^{\f23}(\xi/k-t)\big),\ \text{for} \  k\neq 0,\\
&\mathcal{M}_2(t,k, \xi)=C_{\g}\psi_{1-\delta}(\xi/k-t),\ \text{for} \  k\neq 0,\\
&\mathcal{M}_3(t, k, \xi)=\sum_{j\in \mathbb{Z}\setminus\{0\}}
j^{-1}\lan k-j\ran^{-\delta}\psi_{\delta}\Big(\f{\xi-tj}{\lan k-j\ran+|j|}\Big).
\end{align*}Then $\mathcal{M}_j(t, k, \xi) $ ($j=0,1,2,3$) are well defined and smooth for $(k,\xi)\in(\R\setminus\{0\})\times\R$.

\begin{lemma}\label{lem1}If $(k, \xi),(l, \eta)\in\mathbb{Z}\times\R$, $t\geq0$, $\lan k-l\ran\leq \f{|k|+|l|}{20}$ then 
\begin{align*}
&|\mathcal{M}_1(t, k, \xi)-\mathcal{M}_1(t, l, \eta)|\leq C(|\kappa/k|^\f13+|k|^{-1})|k-l, \xi-\eta-(k-l)t|,\\
&|\mathcal{M}_2(t, k, \xi)-\mathcal{M}_2(t, l, \eta)|\leq C|k|^{-1}|k-l, \xi-\eta-(k-l)t|,\\
&|\mathcal{M}_3(t, k, \xi)-\mathcal{M}_3(t, l, \eta)|\leq C|k|^{-1}{|k-l, \xi-\eta|},\\
&|\mathcal{M}_0(t, k, \xi)-\mathcal{M}_0(t, l, \eta)|\leq C(|\kappa/k|^\f13+|k|^{-1})(|k-l, \xi-\eta-(k-l)t|+|k-l, \xi-\eta|).
\end{align*}\end{lemma}

\begin{proof}
We can write\begin{align*}
&\mathcal{M}_1(t, k, \xi)=\widetilde{\mathcal{M}}_1(k, \xi-kt),\quad  \text{with} \quad
\widetilde{\mathcal{M}}_1(k, \xi)=\psi_1\big(\kappa^{\f13}|k|^{\f23}\xi/k\big),\ \text{for} \  k\neq 0,\\
&\mathcal{M}_2(t, k, \xi)=\widetilde{\mathcal{M}}_2(k, \xi-kt),\quad  \text{with} \quad
\widetilde{\mathcal{M}}_2(k, \xi)=C_{\g}\psi_{1-\delta}(\xi/k),\ \text{for} \  k\neq 0.
\end{align*}
In this case $k\sim l $, $kl> 0$. By symmetry,  we can assume that $k<l$.
 By mean value theorem, there exists   {$(k_j, \xi_j)\in\R^2$ with $ k\leq k_j\leq l$ for $j=1,2,3$},  
 such that 
\begin{align*}
&|\mathcal{M}_1(t, k, \xi)-\mathcal{M}_1(t, l, \eta)|=|\widetilde{\mathcal{M}}_1( k, \xi-kt)-\widetilde{\mathcal{M}}_1( l, \eta-lt)|\\
\leq &|\p_k \widetilde{\mathcal{M}}_1(k_1, \xi_1)||k-l|+|\p_\xi \widetilde{\mathcal{M}}_1(k_1, \xi_1)||\xi-\eta-(k-l)t|\\
\leq& |k_1|^{-1}|k-l|+(\kappa/|k_1|)^\f13|\xi-\eta-(k-l)t|\\
\leq& C(|\kappa/k|^\f13+|k|^{-1})|k-l, \xi-\eta-(k-l)t|,
\end{align*}
where we used
\begin{align*}
&|\p_k \widetilde{\mathcal{M}}_1(k_1, \xi_1)|= |3k_1|^{-1}\kappa^{\f13}|k_1|^{-\f13}|\xi_1|/(1+|\kappa/k_1|^{\f23}|\xi_1|^2)\leq  |k_1|^{-1},\\
&|\p_\xi \widetilde{\mathcal{M}}_1(k_1, \xi_1)|= |\kappa/k_1|^{\f13}/(1+|\kappa/k_1|^{\f23}|\xi_1|^2)\leq |\kappa/k_1|^{\f13}.
\end{align*}
And 
\begin{align*}
&|\mathcal{M}_2(t, k, \xi)-\mathcal{M}_2(t, l, \eta)|=|\widetilde{\mathcal{M}}_2( k, \xi-kt)-\widetilde{\mathcal{M}}_2( l, \eta-lt)|\\
\leq &|\p_k \widetilde{\mathcal{M}}_2(k_2, \xi_2)||k-l|+|\p_\xi \widetilde{\mathcal{M}}_2(k_2, \xi_2)||\xi-\eta-(k-l)t|\\
\leq &\f{C}{|k_2|} \Big(|k-l|+|\xi-\eta-(k-l)t|\Big)\leq C|k|^{-1}|k-l, \xi-\eta-(k-l)t|,\end{align*}
where we used (for $0<\delta<\min(s-3/2,1/2)$)
\begin{align*}
|\p_k \widetilde{\mathcal{M}}_2(k_2, \xi_2)|&\leq  C_\g\f{|\xi_2|}{|k_2|^2}\Big\lan \f{\xi_2}{k_2}\Big\ran ^{-2+\delta}\leq \f{C}{|k_2|}, \\
|\p_\xi \widetilde{\mathcal{M}}_2(k_2, \xi_2)|&\leq \f{C_\g}{|k_2|}\Big\lan \f{\xi_2}{k_2}\Big\ran ^{-2+\delta} \leq \f{C}{|k_2|}.
\end{align*}
And 
\begin{align*}
&|\mathcal{M}_3(t, k, \xi)-\mathcal{M}_3(t, l, \eta)|
\leq|\p_k \mathcal{M}_3(t, k_3, \xi_3)||k-l|+|\p_\xi \mathcal{M}_3(t, k_3, \xi_3)||\xi-\eta|\\
&\leq \f{C}{|k_3|} \Big(|k-l|+|\xi-\eta|\Big)\leq C|k|^{-1}{|k-l, \xi-\eta|},
\end{align*}
where we used ({for} $0<\delta<\min(s-3/2,1/2)$)
\begin{align*}
|\p_k \mathcal{M}_3(t, k_3, \xi_3)|& \leq C\sum_{j\in \mathbb{Z}\setminus\{0\}}
|j|^{-1}\lan k_3-j\ran^{-\delta-1}\bigg|\psi_{\delta}\bigg(\f{\xi_3-tj}{\lan k_3-j\ran+|j|}\bigg)\bigg|\\&\quad +C\sum_{j\in \mathbb{Z}\setminus\{0\}} |j|^{-1}\lan k_3-j\ran^{-\delta}\bigg\lan \f{\xi_3-tj}{\lan k_3-j\ran+|j|}\bigg\ran^{-1-\delta} \f{|\xi_3-tj|}{(\lan k_3-j\ran+|j|)^2} \\
&  \leq  C\sum_{j\in \mathbb{Z}\setminus\{0\}}|j|^{-1}\lan k_3-j\ran^{-\delta-1}+|j|^{-1}\lan k_3-j\ran^{-\delta}(\lan k_3-j\ran+|j|)^{-1}\\
&  \leq  C\sum_{j\in \mathbb{Z}\setminus\{0\}} |j|^{-1}\lan k_3-j\ran^{-\delta-1}
 \leq \f{C}{|k_3|}, 
\end{align*}
and
\begin{align*}
&|\p_\xi \mathcal{M}_3(t, k_3, \xi_3)| =\sum_{j\in \mathbb{Z}\setminus\{0\}}
|j|^{-1}\lan k_3-j\ran^{-\delta}\bigg\lan \f{\xi_3-tj}{\lan k_3-j\ran+|j|}\bigg\ran^{-1-\delta} \f{1}{\lan k_3-j\ran+|j|} \\
  \leq&C\sum_{j\in \mathbb{Z}\setminus\{0\}} \f{|j|^{-1}\lan k_3-j\ran^{-\delta}}{\lan k_3-j\ran+|j|}\leq C\sum_{j\in \mathbb{Z}\setminus\{0\}} |j|^{-1}\lan k_3-j\ran^{-\delta-1}
 \leq \f{C}{|k_3|}.
\end{align*}
{This completes the proof of the first three inequalities. Summing up the first three inequalities and using $\mathcal{M}_0=\mathcal{M}_1+\mathcal{M}_2+\mathcal{M}_3$ gives the fourth inequality.}
\end{proof}

\begin{lemma}\label{lem2} For $(k, \xi),(l, \eta)\in\mathbb{Z}\times\R$, $t\geq T_0=\kappa^{-\f16}$, it holds that
\begin{align}\label{M1}
&|l \mathcal{M}^2(t,k, \xi)-k\mathcal{M}^2(t,l, \eta)|\non\\
\leq& C\mathcal{M}(t, k,\xi)\mathcal{M}(t, l,\eta)\mathcal{M}(t, k-l,\xi-\eta) (\lan k, \xi\ran^{\f12-s}+\lan l, \eta\ran^{\f12-s}+\lan k-l, \xi-\eta\ran^{\f12-s})\times\non\\
&(|k|^{\f13}|k-l|^{\f13}+|l|^{\f13}|k-l|^{\f13}+|kl|^{\f13}+\kappa^{\f23}|kl|).
\end{align} 
Moreover, if $k\neq l$ then
\begin{align}\label{M2}
&|(\eta-lt) \mathcal{M}^2(t, k,\xi)|\leq Ce^{\epsilon \kappa^\f13 t}(|l,\eta-lt|+|k,\xi-kt|)\mathcal{M}(t, k,\xi)\mathcal{M}(t, l,\eta)\non\\
&+C(\kappa^{\f16}|l,\eta-lt|+\kappa^{-\f16}|l|^{\f13}+\kappa^{-\f16}|k|^{\f13})\mathcal{M}(t, k,\xi)\mathcal{M}(t, l,\eta)\mathcal{M}(t, k-l,\xi-\eta)\lan l, \eta\ran^{-s}\non\\
&+C\kappa^{-\f13}\mathcal{M}(t, k,\xi)\mathcal{M}(t, l,\eta)\mathcal{M}(t, k-l,\xi-\eta)\lan l, \eta\ran^{-\f{1+\delta}{2}}| k-l, \xi-\eta-(k-l)t|^{-1/2},
\end{align}
and similarly
\begin{align}\label{M3}
&|(\xi-kt) \mathcal{M}^2(t, l, \eta)|\leq Ce^{\epsilon \kappa^\f13 t}(|l,\eta-lt|+|k,\xi-kt|)\mathcal{M}(t, k,\xi)\mathcal{M}(t, l,\eta)\non\\
&+C(\kappa^{\f16}|k,\xi-kt|+\kappa^{-\f16}|l|^{\f13}+\kappa^{-\f16}|k|^{\f13})\mathcal{M}(t, k,\xi)\mathcal{M}(t, l,\eta)\mathcal{M}(t, k-l,\xi-\eta)\lan k, \xi\ran^{-s}\non\\
&+C\kappa^{-\f13}\mathcal{M}(t, k,\xi)\mathcal{M}(t, l,\eta)\mathcal{M}(t, k-l,\xi-\eta)\lan k,\xi\ran^{-\f{1+\delta}{2}}| k-l, \xi-\eta-(k-l)t|^{-1/2}.
\end{align}
\end{lemma}

\begin{proof}
We first proof \eqref{M1}. Recall that $\mathcal{M}(t, k, \xi)\sim \lan k, \xi-kt\ran^\frac{1}{2} \lan k, \xi\ran^se^{\epsilon\kappa^\f13 t\mathbf{1}_{k\neq 0}}$.

\textbf{Case 1:}
$\langle k-l, \xi-\eta-(k-l)t\rangle+\langle k-l, \xi-\eta\rangle\geq (|k|+|l|)/10$.
In this case, we have
 \begin{align*}
&\f{|l| \mathcal{M}(t, k, \xi)}{\mathcal{M}(t, l, \eta)\mathcal{M}(t, k-l, \xi-\eta)}\\
\leq&\f{|l|\lan k, \xi-kt\rangle^\frac{1}{2}e^{\epsilon\kappa^\f13 t(\mathbf{1}_{k\neq 0}-\mathbf{1}_{l\neq 0}-\mathbf{1}_{k-l\neq 0})} }{\lan l, \eta-lt\rangle^\frac{1}{2}\lan k-l, (\xi-\eta)-(k-l)t\rangle^\frac{1}{2} } \f{\lan k, \xi\ran^s }{\lan l, \eta\ran^s\lan k-l, \xi-\eta\ran^s} \\
\leq&\bigg(\f{C|l|  }{ \lan l, \eta-lt\ran^\frac{1}{2}}+\f{C|l|  }{\lan k-l, \xi-\eta-(k-l)t\ran^\frac{1}{2} } \bigg)\bigg(\f{1}{\lan l, \eta\ran^s}+\f{1}{\lan k-l, \xi-\eta\ran^s}\bigg)\\
\leq&  
\f{C|l|}{\lan l, \eta\ran^s}+C\bigg(|l|^\frac{1}{2}+\f{|l|^\frac{1}{2}(\langle k-l, \xi-\eta-(k-l)t\rangle^\frac{1}{2}+\langle k-l, \xi-\eta\rangle^\frac{1}{2})  }{\lan k-l, \xi-\eta-(k-l)t\ran^\frac{1}{2} } \bigg)\f{1}{\lan k-l, \xi-\eta\ran^s}\\
\leq& \f{C|l|}{\lan l, \eta\ran^s}+\f{C|l|^\frac{1}{2}}{\lan k-l, \xi-\eta\ran^s}+\f{C|l|^\frac{1}{2}}{\lan k-l, \xi-\eta\ran^{s-1/2}}\leq C|l|^\frac{1}{2}(\lan l, \eta\ran^{\frac{1}{2}-s}+\lan k-l, \xi-\eta\ran^{\frac{1}{2}-s}),
\end{align*}
where we used $\mathbf{1}_{k\neq 0}-\mathbf{1}_{l\neq 0}-\mathbf{1}_{k-l\neq 0}\leq0 $  and
 \begin{align}\label{f1}
&\lan k, \xi-kt\ran\leq \lan l, \eta-lt\rangle+\lan k-l, \xi-\eta-(k-l)t\ran\non\\
&\leq  \lan l, \eta-lt\rangle\lan k-l, \xi-\eta-(k-l)t\ran\Big(\f{1}{\lan k-l, \xi-\eta-(k-l)t\ran}+\f{1}{\lan l, \eta-lt\rangle}\Big),\non\\
&\lan k, \xi\ran
\leq\lan l, \eta\ran+\lan k-l, \xi-\eta\ran\leq  \lan l, \eta\ran \lan k-l, \xi-\eta\ran\Big(\f{1}{\lan k-l, \xi-\eta\ran}+\f{1}{\lan l, \eta\ran}\Big).
\end{align}
Similarly 
\begin{align*}
\f{|k| \mathcal{M}(t, l, \eta)}{\mathcal{M}(t, k, \xi)\mathcal{M}(t, k-l, \xi-\eta)}
\leq C|k|^\frac{1}{2}(\lan k, \xi\ran^{\frac{1}{2}-s}+\lan k-l, \xi-\eta\ran^{\frac{1}{2}-s}).
\end{align*}
We  can get  
\begin{align*}
&\f{|l \mathcal{M}^2(t, k, \xi)-k\mathcal{M}^2(t, l, \eta)|}{\mathcal{M}(t, k, \xi) \mathcal{M}(t, l, \eta)\mathcal{M}(t, k-l, \xi-\eta) }\\
\leq &\f{|l| \mathcal{M}(t, k, \xi)}{\mathcal{M}(t, l, \eta)\mathcal{M}(t, k-l, \xi-\eta)}+\f{|k| \mathcal{M}(t, l, \eta)}{\mathcal{M}(t, k, \xi)\mathcal{M}(t, k-l, \xi-\eta)}\\
\leq& C(|k|^\frac{1}{2}+|l|^\frac{1}{2})(\lan k, \xi\ran^{\f12-s}+\lan l, \eta\ran^{\f12-s}+\lan k-l, \xi-\eta\ran^{\frac{1}{2}-s})
\end{align*}
without using symmetry. Then \eqref{M1} follows by 
\begin{align*}
|k|^\frac{1}{2}+|l|^\frac{1}{2}\leq |k|^\frac{2}{3}+|l|^\frac{2}{3}\leq |k|^{\frac{1}{3}}(|l|^{\frac{1}{3}}+|k-l|^{\frac{1}{3}})+|l|^{\frac{1}{3}}(|k|^{\frac{1}{3}}+|k-l|^{\frac{1}{3}}).
\end{align*}

\textbf{Case 2:}  
$\langle k-l, \xi-\eta-(k-l)t\rangle+\langle k-l, \xi-\eta\rangle\leq (|k|+|l|)/10$.
(In this case $\langle k-l\rangle\leq (|k|+|l|)/20$, $\langle k,\xi\rangle\sim\langle l,\eta\rangle\geq\langle k-l, \xi-\eta\rangle $, 
$\lan k, \xi-kt\ran\sim\lan l, \eta-lt\ran$, $1\leq|k|\sim|l|$, 
$k\neq0$, $l\neq0$,  $ \mathcal{A}(t, k, \xi)=\mathcal{A}(t, l, \eta)=e^{\epsilon \kappa^\f13 t+\kappa^{-\f13} t^{-2}}=:\mathcal{A}_*$.)


Recall that $\mathcal{M}(t, k, \xi)= |k_+, \xi-kt|^\frac{1}{2} \lan k, \xi\ran^s\mathcal{A} e^{\mathcal{M}_0}$. We divide
\begin{align}\label{T''}
&|l  \mathcal{M}^2(t, k, \xi)-k \mathcal{M}^2(t, l, \eta)|\\ \non
= &\mathcal{A}_*^2\Big| l|k_+, \xi-kt| \lan k, \xi\ran^{2s} e^{2 \mathcal{M}_0(t, k, \xi)}-k |l_+, \eta-lt| \lan l, \eta\ran^{2s} e^{2 \mathcal{M}_0(t, l, \eta)}\Big|\leq  \mathcal{A}_*^2 ( \cT^*_1+\cT^*_2+\cT^*_3),
\end{align}
where 
\begin{align*}
&\cT^*_1= |k_+, \xi-kt|\lan k, \xi\ran^{2s}\Big| l e^{2  \mathcal{M}_0 (t, k, \xi)}-k e ^{2  \mathcal{M}_0 (t, l, \eta)}\Big|,\\
&\cT^*_2=\Big||k_+, \xi-kt| -|l_+, \eta-lt|\Big| \lan k, \xi\ran^{2s} |k|  e ^{2  \mathcal{M}_0 (t, l, \eta)},\\
&\cT^*_3= |k|e ^{2  \mathcal{M}_0 (t, l, \eta)}|l_+, \eta-lt| \Big|\lan l, \eta\ran^{2s}- \lan k, \xi\ran^{2s}\Big|.
\end{align*}
Then we have
\begin{align*}
\cT^*_1&=|k_+, \xi-kt|\lan k, \xi\ran^{2s} \Big| l e^{2  \mathcal{M}_0 (t, k, \xi)}-k e^{2  \mathcal{M}_0 (t, l, \eta)}\Big|\\
&\leq |k_+, \xi-kt|\lan k, \xi\ran^{2s}(|k-l|e^{2  \mathcal{M}_0 (t, k, \xi)}+|k||e^{2  \mathcal{M}_0 (t, k, \xi)}-e^{2 \mathcal{M}_0 (t, l, \eta)}|)\\
&\leq C\lan k, \xi-kt\ran\lan k, \xi\ran^{2s}(|k-l|+|k|| \mathcal{M}_0 (t, k, \xi)- \mathcal{M}_0 (t, l, \eta)|),\end{align*}
and
\begin{align*}
\cT^*_2=\Big||k_+, \xi-kt| -|l_+, \eta-lt|\Big| \lan k, \xi\ran^{2s} |k|  e^{2  \mathcal{M}_0 (t, l, \eta)}
\leq C\lan k-l, (\xi-\eta)-(k-l)t\ran \lan k, \xi\ran^{2s} |k|,
\end{align*}and\begin{align*}
 \cT^*_3&=|k|e ^{2  \mathcal{M}_0 (t, l, \eta)}|l_+, \eta-lt| \Big|\lan l, \eta\ran^{2s}- \lan k, \xi\ran^{2s}\Big|\\
&\leq C |k|\lan l, \eta-lt\ran\Big|\lan l, \eta\ran^{2s}- \lan k, \xi\ran^{2s}\Big|\\
&\leq C|k|\lan l, \eta-lt\ran\lan k-l, \xi-\eta\ran\lan k, \xi\ran^{2s-1}\leq C\lan l, \eta-lt\ran\lan k-l, \xi-\eta\ran\lan k, \xi\ran^{2s}.
\end{align*}
Summing up, by \eqref{T''} and Lemma \ref{lem1},  we have 
\begin{align*}
&\f{|l  \mathcal{M}^2(t, k,  \xi)-k \mathcal{M}^2(t, l, \eta)|}{{ \mathcal{A}_*^2}\lan k, \xi-kt\ran^{\f12}\lan k, \xi\ran^{s}\lan l, \eta-lt\ran^{\f12}\lan l, \eta\ran^{s}\lan k-l, (\xi-\eta)-(k-l)t\ran^{\f12}\lan k-l, \xi-\eta\ran^{s}}\\
\leq& \f{C\lan k, \xi-kt\ran^{\f12}\lan k, \xi\ran^s |k-l|}{\lan l, \eta-lt\ran^{\f12}\lan l, \eta\ran^s\lan k-l, \xi-\eta\ran^{s}\lan k-l, (\xi-\eta)-(k-l)t\ran^{\f12}}\\
&+ \f{C| k|{\lan k, \xi\ran^s} \lan k-l, (\xi-\eta)-(k-l)t\ran^{\f12}}{\lan k, \xi-kt\ran^{\f12}\lan l, \eta-lt\ran^{\f12}\lan l, \eta\ran^s\lan k-l, \xi-\eta\ran^{s}}\\
&+\f{C|k|\lan k, \xi-kt\ran^\f12\lan k, \xi\ran^s}{\lan l, \eta-lt\ran^\f12  \lan k-l, (\xi-\eta)-(k-l)t\ran^{\f12}\lan l, \eta\ran^s\lan k-l, \xi-\eta\ran^{s}}
|M_0 (t, k, \xi)-M_0 (t, l, \eta)|\\
& +\f{C\lan l, \eta-lt\ran^\f12\lan k, \xi\ran^{s}}{ \lan k, \xi-kt\ran^\f12 \lan k-l, (\xi-\eta)-(k-l)t\ran^{\f12}\lan l, \eta\ran^s\lan k-l, \xi-\eta\ran^{s-1}} \\
\leq&  C\big(|k-l|+\lan k-l, (\xi-\eta)-(k-l)t\ran^{\f12}\big)\lan k-l, \xi-\eta\ran^{-s}\\
&+C\f{|k|(|\nu/k|^\f13+|k|^{-1})(|k-l, \xi-\eta-(k-l)t|+|k-l, \xi-\eta|)}{\lan k-l, (\xi-\eta)-(k-l)t\ran^{\f12}\lan k-l, \xi-\eta\ran^s}\\& +C {\lan k-l, \xi-\eta\ran^{1-s}} \\
\leq&C\big(|k-l|+(|k|+|l|)^{\f12}\big)\lan k-l, \xi-\eta\ran^{-s}+C{\lan k-l, \xi-\eta\ran^{1-s}}\\
&+C(|\nu k^2|^\f13+1)(|k-l, \xi-\eta-(k-l)t|+|k-l, \xi-\eta|)^{\f12}\lan k-l, \xi-\eta\ran^{\f12-s}\\
\leq& C\lan k-l, \xi-\eta\ran^{1-s}+C(|\nu k^2|^\f13+1)(|k|+|l|)^{\f12}\lan k-l, \xi-\eta\ran^{\f12-s}\\
\leq& C(1+|\nu k^2|^\f13)(|k|+|l|)^{\f12}\lan k-l, \xi-\eta\ran^{\f12-s}\\
\leq &C(1+|\nu k^2|^\f23)(|k|+|l|)^{\f23}\lan k-l, \xi-\eta\ran^{\f12-s}\\
\leq &C(|l|^\f13|k|^\f13+\kappa^{\f23} |kl|)\lan k-l, \xi-\eta\ran^{\f12-s}.
\end{align*}
Here we used  \begin{align*}
 &\langle k,\xi\rangle\sim\langle l,\eta\rangle, \quad
\lan k, \xi-kt\ran\sim\lan l, \eta-lt\ran,\quad |k|\sim|kl|^{\f12}\leq \lan k, \xi-kt\ran^{\f12}\lan l, \eta-lt\ran^{\f12},\\
&|k-l, \xi-\eta-(k-l)t|+|k-l, \xi-\eta|\leq \lan k-l, (\xi-\eta)-(k-l)t\ran\lan k-l, \xi-\eta\ran,\\
&(|k-l, \xi-\eta-(k-l)t|+|k-l, \xi-\eta|)^{\f12}\leq (|k|+|l|)^{\f12},\\
&\lan k-l, \xi-\eta\ran^{1-s}\leq (|k|+|l|)^{\f12}\lan k-l, \xi-\eta\ran^{\f12-s}.\end{align*}
Then  \eqref{M1} follows from 
 \begin{align*}
 &\mathcal{M}(t, k,  \xi)=|k_+, \xi-kt|^\frac{1}{2} \lan k, \xi\ran^s  \mathcal{A}_* e^{ \mathcal{M}_0(t, k, \xi)}\sim\mathcal{A}_*|k_+, \xi-kt|^\frac{1}{2} \lan k, \xi\ran^s\sim\mathcal{A}_*\lan k, \xi-kt\ran^\frac{1}{2} \lan k, \xi\ran^s,\\
 &\mathcal{M}(t, l,  \eta)\sim\mathcal{A}_*\lan l, \eta-lt\ran^\frac{1}{2} \lan l,  \eta\ran^s,\\
&\mathcal{M}(t, k-l,  \xi-\eta)\geq|(k-l)_+, \xi-\eta-(k-l)t|^\frac{1}{2} \lan k-l,  \xi-\eta\ran^s   e^{ \mathcal{M}_0(t, k-l,  \xi-\eta)}\\
&\sim\lan k-l, \xi-\eta-(k-l)t\ran^\frac{1}{2} \lan k-l,  \xi-\eta\ran^s.\end{align*}

 Now we consider  \eqref{M2}. {We assume $k\neq l$, then $ (k-l)_+=|k-l|$.}
 
 \textbf{Case 1:}
$\langle k-l,\xi-\eta\rangle\geq (\lan k, \xi\ran+\langle l,\eta\rangle)/4$.

Recall that $\mathcal{M}(t, k, \xi)\sim \lan k, \xi-kt\ran^\frac{1}{2} \lan k, \xi\ran^se^{\epsilon\kappa^\f13 t\mathbf{1}_{k\neq 0}}$, $\mathbf{1}_{k\neq 0}\leq \mathbf{1}_{k-l\neq 0}=1$, we have
 \begin{align*}
&\f{|\eta-lt| \mathcal{M}(t, k, \xi)}{\mathcal{M}(t, l, \eta)\mathcal{M}(t, k-l, \xi-\eta)}\\
\leq&\f{C|\eta-lt|\lan k, \xi-kt\ran^\frac{1}{2} e^{-\epsilon \kappa^\f13 t\mathbf{1}_{l\neq0}}}{\lan l, \eta-lt\ran^\frac{1}{2}\lan k-l, \xi-\eta-(k-l)t\ran^\frac{1}{2} } \f{\lan k, \xi\ran^s }{\lan l, \eta\ran^s\lan k-l, \xi-\eta\ran^s} \\
\leq&C \bigg(\f{|\eta-lt|}{\lan l, \eta-lt\ran^\frac{1}{2}} +\f{|\eta-lt|e^{-\epsilon \kappa^\f13 t\mathbf{1}_{l\neq0}} }{|(k-l)_+, (\xi-\eta)-(k-l)t|^\frac{1}{2} } \bigg)\f{1}{\lan l, \eta\ran^s}\\
\leq & C\bigg(|\eta-lt|^\frac{1}{2} +\f{\kappa^{-\f13}\lan l, \eta\ran }{\lan k-l, \xi-\eta-(k-l)t\ran^\frac{1}{2} } \bigg) \f{1}{\lan l, \eta\ran^s} \non\\
= & \f{C|\eta-lt|^\frac{1}{2}}{\lan l, \eta\ran^s}+\f{C\kappa^{-\f13}}{|k-l, (\xi-\eta)-(k-l)t|^\frac{1}{2}\lan l, \eta\ran^{s-1}}\\
\leq & \f{C\kappa^{\f16}|\eta-lt|+\kappa^{-\f16}(|k|^{\f13}+|l|^{\f13})}{\lan l, \eta\ran^s}+\f{C\kappa^{-\f13}}{|k-l, (\xi-\eta)-(k-l)t|^\frac{1}{2}\lan l, \eta\ran^{\f{1+\delta}{2}}},
\end{align*}
where we use (as $k\neq l$), $s-1>1/2+\delta>\f{1+\delta}{2}$,
\begin{align*}
&|\eta-lt|\leq \lan l, \eta\ran(1+t\mathbf{1}_{l\neq0})\leq C\kappa^{-\f13}\lan l, \eta\ran e^{\epsilon \kappa^\f13 t\mathbf{1}_{l\neq0}},\quad 1\leq |k-l|^{\f13}\leq |k|^{\f13}+|l|^{\f13},\\
&\lan k-l, \xi-\eta-(k-l)t\ran\sim|k-l, (\xi-\eta)-(k-l)t|. 
\end{align*}

 \textbf{Case 2:}
$\langle k-l,\xi-\eta\rangle\leq (\langle k,\xi\rangle+\langle l,\eta\rangle)/4$. (In this case $\langle k,\xi\rangle\sim\langle l,\eta\rangle $.)

Recall that $\mathcal{M}(t, k, \xi)\sim \lan k, \xi-kt\ran^\frac{1}{2} \lan k, \xi\ran^se^{\epsilon\kappa^\f13 t\mathbf{1}_{k\neq 0}}$, $\mathbf{1}_{k\neq 0}- \mathbf{1}_{l\neq 0}\leq1$, we have
 \begin{align*}
&\f{|\eta-lt| \mathcal{M}(t, k, \xi)}{\mathcal{M}(t, l, \eta)}e^{-\epsilon \kappa^\f13 t}
\leq C \f{|\eta-lt|\lan k, \xi-kt\ran^\frac{1}{2} \lan k, \xi\ran^s }{\lan l, \eta-lt\ran^\frac{1}{2} \lan l, \eta\ran^s}  
\leq C |\eta-lt|^\frac{1}{2}\lan k, \xi-kt\ran^\frac{1}{2} \\
\leq& C (|\eta-lt|+\lan k, \xi-kt\ran) \leq C (|\eta-lt|+| k, \xi-kt|+1) \\
\leq& C (|\eta-lt|+| k, \xi-kt|+|k|+|l|) \leq  C (|l,\eta-lt|+| k, \xi-kt|) ,
\end{align*}
where we used $1\leq |k-l|\leq |k|+|l|$. Combining the 2 cases, we obtain  \eqref{M2}.

By symmetry (switching $(k, \xi)$ and $(l, \eta)$) we have \eqref{M3}.
\end{proof}

\begin{lemma}\label{lem5}
{ Assume $ \hat{f}(0,\xi)=0$}, then it holds that
\begin{align*}
&\sum_{k, l\in \mathbb{Z}}\int_{\mathbb{R}^2}
 |\hat{f}(k-l, \xi-\eta )||\hat{g}(l, \eta)||{\hat{h}}(k, \xi)|\lan l, \eta\ran^{-\f{1+\delta}{2}}|k-l|^{\f{\delta}{2}}| k-l, \xi-\eta-(k-l)t|^{-\f{1+\delta}{2}} d\xi d\eta\\
 \leq& C\|\hat{f}\|_{L^2}\|\hat{g}\|_{L^2}\|\sqrt{\Upsilon(t,k,\xi)}\hat{h}(k, \xi)\|_{L_{k,\xi}^2}.
\end{align*}\end{lemma}
\begin{proof}
We use H\"older's inequality to get
\begin{align*}
&\sum_{k, l\in \mathbb{Z}}\int_{\mathbb{R}^2}
 |\hat{f}(k-l, \xi-\eta )||\hat{g}(l, \eta)||{\hat{h}}(k, \xi)|\lan l, \eta\ran^{-\f{1+\delta}{2}}|k-l|^{\f{\delta}{2}}| k-l, \xi-\eta-(k-l)t|^{-\f{1+\delta}{2}} d\xi d\eta\\
\leq& (\sum_{k, l\in \mathbb{Z}}\int_{\mathbb{R}^2}
 |\hat{f}(k-l, \xi-\eta )|^2|\hat{g}(l, \eta)|^2d\xi d\eta)^\f12 \\
 &\quad (\sum_{k, l\in \mathbb{Z}}\int_{\mathbb{R}^2}| {\hat{h}}(k, \xi)|^2 \lan l, \eta\ran^{-(1+\delta)}|k-l|^{\delta}| k-l, \xi-\eta-(k-l)t|^{-(1+\delta)} d\xi d\eta)^\f12\\
 \leq& C\|\hat{f}\|_{L^2}\|\hat{g}\|_{L^2}\sum_{k\in \mathbb{Z}}\int_{\mathbb{R}}| {\hat{h}}(k, \xi)|^2\sum_{l\in \mathbb{Z}\setminus\{k\}}\f{\lan l\ran^{-\delta}(\lan l\ran+ |k-l|)^{\delta}}{|\lan l\ran +|k-l|,\xi-(k-l)t |^{1+\delta}}d\xi,
 \end{align*}
where we used Lemma \ref{Lem4} (with $a=\lan l\ran, b=|k-l|, z=\xi-(k-l)t, \lambda=\delta$) to
deduce that
\begin{equation*}
   \int_\R \f{|k-l|^{\delta} d\eta}{|\lan l\ran,\eta|^{1+\delta}||k-l|,(\xi-\eta)-(k-l)t|^{1+\delta}}\leq\f{C\lan l\ran^{-\delta}(\lan l\ran+ |k-l|)^{\delta}}{|\lan l\ran +|k-l|,\xi-(k-l)t |^{1+\delta}}.
\end{equation*}We also used $|\lan l\ran,\eta|=\lan l, \eta\ran$. Then the result follows from the definition of $ \Upsilon$ in \eqref{Up}.
\end{proof}
{ Based on \eqref{M2}, \eqref{M3}, Lemma \ref{lem5}, Young's convolution inequality and H\"older's inequality, we have the following Lemma.}
\begin{lemma}\label{lem7}
Assume $ \hat{f}(0,\xi)=0$, $t\geq T_0=\kappa^{-\f16}$, then 
\begin{align*}
&\sum_{k, l\in \mathbb{Z}}\int_{\mathbb{R}^2}
 |\hat{f}(k-l, \xi-\eta )||\hat{g}(l, \eta)||{\hat{h}}(k, \xi)|(|(\eta-lt) \mathcal{M}^2(t, k,\xi)|+|(\xi-kt) \mathcal{M}^2(t, l,\eta)|) d\xi d\eta\\
 \leq& Ce^{\epsilon \kappa^\f13 t}\|\hat{f}\|_{L^1}(\||k,\xi-kt|\mathcal{M}\hat{g}\|_{L^2}\|\mathcal{M}\hat{h}\|_{L^2}+
 \|\mathcal{M}\hat{g}\|_{L^2}\||k,\xi-kt|\mathcal{M}\hat{h}\|_{L^2})\\
 &+C\|\mathcal{M}\hat{f}\|_{L^2}(\kappa^{\f16}\||k,\xi-kt|\mathcal{M}{\hat{g}}\|_{L^2}
 +\kappa^{-\f16}\||k|^{\f13}\mathcal{M}{\hat{g}}\|_{L^2})\|\mathcal{M}\hat{h}\|_{L^2}\\
 &+C\|\mathcal{M}\hat{f}\|_{L^2}\|\mathcal{M}\hat{g}\|_{L^2}(\kappa^{\f16}\||k,\xi-kt|\mathcal{M}{\hat{h}}\|_{L^2}
 +\kappa^{-\f16}\||k|^{\f13}\mathcal{M}{\hat{h}}\|_{L^2})\\
 &+C\kappa^{-\f13}\|\lan \xi/k-t\ran^{\f{\delta}{2}}\mathcal{M}\hat{f}\|_{L^2_{k, \xi}}(\|\mathcal{M}\hat{g}\|_{L^2}\|\sqrt{\Upsilon}\mathcal{M}\hat{h}\|_{L^2}+
 \|\mathcal{M}\hat{h}\|_{L^2}\|\sqrt{\Upsilon}\mathcal{M}\hat{g}\|_{L^2}).
\end{align*}\end{lemma}

\setcounter{equation}{0}

\section{Energy estimate in long time scale $t>\kappa^{-\frac{1}{6}}$}
The system for the weighted variables  $(\mathcal{M} \hat{u}, \mathcal{M} \hat{\vth})$ read as
\begin{equation}\label{ub3}
 \left\{\begin{array}{l}
(\mathcal{M} \hat{u})_t-q^* \mathcal{M} \hat{u}+\nu(k^2+(\xi-kt)^2) \mathcal{M} \hat{u}\\
\quad-(\f{k^2-(\xi-kt)^2}{k^2+(\xi-kt)^2}, \f{2k(\xi-kt)}{k^2+(\xi-kt)^2})\mathcal{M} \hat{u}_2+\g^2(-\f{k(\xi-kt)}{k^2+(\xi-kt)^2}, \f{k^2}{k^2+(\xi-kt)^2}) \mathcal{M} \hat{\vth} =\mathcal{M}  \hat{F},\\
(\mathcal{M} \hat{\vth})_t-q^*\mathcal{M} \hat{\vth}+\mu(k^2+(\xi-kt)^2) \mathcal{M} \hat{\vth}-\mathcal{M} \hat{u}_2=\mathcal{M} \hat{G},\\
k \mathcal{M} \hat{u}_1+(\xi-kt)  \mathcal{M} \hat{u}_2=0,
\end{array}\right.
\end{equation}
here we define $q^*:=\f{\p_t M}{M}$, then
\begin{align}\label{def:q*}
q^*
=&-\f{ k (\xi-kt)}{2(k^2+(\xi-kt)^2)}+\epsilon {\kappa}^\f13\mathbf{1}_{k\neq0}-2{\kappa}^{-\f13} t^{-3}-\f{{\kappa}^{1/3}|k|^{4/3}}{|k|^{2/3}+{\kappa}^{2/3}(\xi-kt)^2}\non\\
&-C_{\g}\lan\xi/k-t\ran^{-2+\delta}- \Upsilon.
\end{align}
Here we use the convention that $\lan\xi/k-t\ran^{-2+\delta}|_{k=0}=0 $,
$ \f{{\kappa}^{1/3}|k|^{4/3}}{|k|^{2/3}+{\kappa}^{2/3}(\xi-kt)^2}|_{k=0}=0$.

\begin{proposition}\label{prop: long-time i}
     If $\nu, \mu\in(0, 1)$, $s>\f32$, $\f{\nu+\mu}{2\g\sqrt{\nu \mu} }< 2-\e$, $0<\e<1/\g$,  it holds  that
\begin{align}\label{t big}
        &\f{d}{dt}\sum_{k\in \mathbb{Z}}\int_{\mathbb{R}} \Big(\big|\mathcal{M}\hat{u}\big|^2+\g^2\big|\mathcal{M}\hat{\vth}\big|^2+\Re (\mathcal{M}\hat{\vth}\mathcal{M} \overline{\hat{u}}_1)\Big) d\xi 
+\f{2}{\g}\sum_{k\in \mathbb{Z}\setminus\{0\}}\int_{\mathbb{R}}  |\lan \xi/k-t\ran^{\f{\delta}{2}}\mathcal{M} \hat{u}_2|^2 d\xi\nonumber\\
&+\f{\e}{4} \sum_{k\in \mathbb{Z}}\int_{\mathbb{R}}\kappa^{\f13}|k|^{\f23}(  \big|\mathcal{M} \hat{u}\big|^2+\g^2\big|\mathcal{M} \hat{\vth}\big|^2)d\xi
\nonumber\\
&+\e\sum_{k\in \mathbb{Z}}\int_{\mathbb{R}}\Upsilon(  |\mathcal{M} \hat{u}|^2+\g^2|\mathcal{M} \hat{\vth}|^2) d\xi+
\e\kappa^{-\f13}t^{-3}\sum_{k\in \mathbb{Z}}\int_{\mathbb{R}}(  |\mathcal{M} \hat{u}|^2+\g^2|\mathcal{M} \hat{\vth}|^2) d\xi\nonumber\\
&+\f{\e}{2} \sum_{k\in \mathbb{Z}}\int_{\mathbb{R}} (k^2+(\xi-kt)^2)
\kappa(\big|\mathcal{M}\hat{u}\big|^2+\gamma^2\big|\mathcal{M}\hat{\vth}\big|^2) d\xi \nonumber\\
         \leq &C\|\mathcal{M}(\hat{u},\hat{\vth})\|_{L^2_{k, \xi}}(\||k|^{\f13}\mathcal{M}(\hat{u},\hat{\vth})\|_{L^2_{k, \xi}}^2+\kappa^{\f23}\||k,\xi-kt|\mathcal{M}(\hat{u},\hat{\vth})\|_{L^2_{k, \xi}}^2+\kappa^{-\f23}t^{-3}\|\mathcal{M}\hat{u}\|_{L^2_{k, \xi}}^2\non\\&+
 \kappa^{-\f13}\|\lan \xi/k-t\ran^{\f{\delta}{2}}\mathcal{M}\hat{u}_2\|_{L^2_{k, \xi}}^2+\kappa^{-\f13}\|\sqrt{\Upsilon}\mathcal{M}(\hat{u},\hat{\vth})\|_{L_{k,\xi}^2}^2).
\end{align}

  \end{proposition}
 \begin{proof}
 Multiplying   $ \eqref{ub3}_1$ with $2 \mathcal{M} \overline{\hat{u}}$,  and  then integrating   with respect to $\xi$, subsequently summing over $k$, taking the real part, we get
\begin{align}
& \f{d}{dt}\sum_{k\in \mathbb{Z}}\int_{\mathbb{R}} |\mathcal{M} \hat{u}|^2d\xi-2\sum_{k\in \mathbb{Z}}\int_{\mathbb{R}} q^*\big|\mathcal{M} \hat{u}\big|^2d\xi+2\nu\sum_{k\in \mathbb{Z}}\int_{\mathbb{R}}  (k^2+(\xi-kt)^2)\big|\mathcal{M} \hat{u}\big|^2 d\xi \nonumber\\
&-2\Re\sum_{k\in \mathbb{Z}}\int_{\mathbb{R}}\f{k^2-(\xi-kt)^2}{k^2+(\xi-kt)^2}\mathcal{M} \hat{u}_2\mathcal{M} \overline{\hat{u}}_1d\xi +2\g^2\Re\sum_{k\in \mathbb{Z}}\int_{\mathbb{R}}\f{(\xi-kt)^2}{k^2+(\xi-kt)^2}\mathcal{M} \hat{\vth} \mathcal{M} \overline{\hat{u}}_2d\xi\nonumber\\
&+2\Re\sum_{k\in \mathbb{Z}}\int_{\mathbb{R}}\f{2k^2}{k^2+(\xi-kt)^2}\mathcal{M} \hat{u}_2\mathcal{M} \overline{\hat{u}}_1d\xi +2\g^2\Re\sum_{k\in \mathbb{Z}}\int_{\mathbb{R}}\f{k^2}{k^2+(\xi-kt)^2}\mathcal{M} \hat{\vth} \mathcal{M} \overline{\hat{u}}_2d\xi\nonumber\\
=&2\Re\sum_{k\in \mathbb{Z}}\int_{\mathbb{R}} \mathcal{M} \hat{F}  \mathcal{M} \overline{\hat{u}}.\label{u1*}
\end{align}
Multiplying   $ \eqref{ub3}_2$ with $2 \mathcal{M} \overline{\hat{\vth}}$,  and  then integrating   with respect to $\xi$, subsequently summing over $k$, taking the real part, we get
\begin{align}
& \frac{d}{dt}\sum_{k\in \mathbb{Z}}\int_{\mathbb{R}}\big| \mathcal{M} \hat{\vth}\big|^2d\xi-2\sum_{k\in \mathbb{Z}}\int_{\mathbb{R}} q^*\big|\mathcal{M} \hat{\vth}\big|^2d\xi+2\mu \sum_{k\in \mathbb{Z}}\int_{\mathbb{R}} (k^2+(\xi-kt)^2)\big|\mathcal{M} \hat{\vth}\big|^2d\xi\nonumber\\
&-2\Re\sum_{k\in \mathbb{Z}}\int_{\mathbb{R}}\mathcal{M} \hat{u}_2 \mathcal{M} \overline{\hat{\vth}}d\xi 
=2\Re\sum_{k\in \mathbb{Z}}\int_{\mathbb{R}}\mathcal{M} \hat{G} \mathcal{M} \overline{\hat{\vth}}d\xi.\label{vth*}
\end{align}
Combining  \eqref{u1*} and \eqref{vth*} together, we have
\begin{align}\label{estimate-3*}
& \frac{d}{dt}\sum_{k\in \mathbb{Z}}\int_{\mathbb{R}}\big| \mathcal{M} \hat{u}\big|^2+\gamma^2\big| \mathcal{M} \hat{\vth}\big|^2d\xi-2\sum_{k\in \mathbb{Z}}\int_{\mathbb{R}} q^*(\big|\mathcal{M} \hat{u}\big|^2+\gamma^2\big|\mathcal{M} \hat{\vth}\big|^2)d\xi\nonumber\\
&+2 \sum_{k\in \mathbb{Z}}\int_{\mathbb{R}} (k^2+(\xi-kt)^2)(\nu\big|\mathcal{M} \hat{u}\big|^2+\mu\gamma^2\big|\mathcal{M} \hat{\vth}\big|^2) d\xi+2 \Re\sum_{k\in \mathbb{Z}}\int_{\mathbb{R}}\mathcal{M} \hat{u}_2\mathcal{M} \overline{\hat{u}}_1 d\xi \nonumber\\
=&2\Re\sum_{k\in \mathbb{Z}}\int_{\mathbb{R}}\mathcal{M} \hat{F} \mathcal{M} \overline{\hat{u}}d\xi+2\g^2\Re\sum_{k\in \mathbb{Z}}\int_{\mathbb{R}}  \mathcal{M} \hat{G} \mathcal{M} \overline{\hat{\vth}} d\xi.
\end{align}
Additionally, by $\lan \f{\xi}{k}-t\ran^{-2}\big|\hat{u}\big|^2=\big|\hat{u}_2\big|^2$ (see \eqref{u2}), we have
\begin{align}\label{linear L*}
-2\sum_{k\in \mathbb{Z}}\int_{\mathbb{R}} q^*\big|\mathcal{M} \hat{u}\big|^2d\xi&=-\Re\sum_{k\in \mathbb{Z}}\int_{\mathbb{R}}\mathcal{M} \hat{u}_2\mathcal{M} \overline{\hat{u}}_1 d\xi+\sum_{k\in \mathbb{Z}}\int_{\mathbb{R}} (4\kappa^{-\f13} t^{-3}-2\epsilon \kappa^\f13\mathbf{1}_{k\neq0})\big|\mathcal{M} \hat{u}\big|^2d\xi\non\\
&+2\sum_{k\in \mathbb{Z}}\int_{\mathbb{R}} \f{{\kappa}^{1/3}|k|^{4/3}}{|k|^{2/3}+{\kappa}^{2/3}(\xi-kt)^2}\big|\mathcal{M} \hat{u}\big|^2d\xi\non\\
&+2 C_\g\sum_{k\in \mathbb{Z}\setminus\{0\}}\int_{\mathbb{R}}  |\lan \xi/k-t\ran^{\f{\delta}{2}}\mathcal{M} \hat{u}_2|^2 d\xi+2 \sum_{k\in \mathbb{Z}}\int_{\mathbb{R}} \Upsilon |\mathcal{M} \hat{u}|^2 d\xi,
\end{align}
and
\begin{align}\label{linear L**}
&-2\sum_{k\in \mathbb{Z}}\int_{\mathbb{R}} q^*\g^2\big|\mathcal{M} \hat{\vth}\big|^2 d\xi\non\\
&=\Re\sum_{k\in \mathbb{Z}}\int_{\mathbb{R}}\f{ \g^2 k (\xi-kt)}{k^2+(\xi-kt)^2} \big|\mathcal{M} \hat{\vth}\big|^2d\xi+\sum_{k\in \mathbb{Z}}\int_{\mathbb{R}} (4\kappa^{-\f13} t^{-3}-2\epsilon \kappa^\f13\mathbf{1}_{k\neq0})\g^2\big|\mathcal{M} \hat{\vth}\big|^2d\xi\non\\
&+\sum_{k\in \mathbb{Z}}\int_{\mathbb{R}}\f{2{\kappa}^{1/3}|k|^{4/3}}{|k|^{2/3}+{\kappa}^{2/3}(\xi-kt)^2}\g^2\big|\mathcal{M} \hat{\vth}\big|^2d\xi\non\\
&+2 C_\g\sum_{k\in \mathbb{Z}\setminus\{0\}}\int_{\mathbb{R}} \g^2 |\lan \xi/k-t\ran^{-1+\f{\delta}{2}}\mathcal{M} \hat{\vth}|^2 d\xi
+2 \g^2\sum_{k\in \mathbb{Z}}\int_{\mathbb{R}} \Upsilon  |\mathcal{M} \hat{\vth}|^2 d\xi.
\end{align}

So \eqref{estimate-3*} can be rewritten as 
\begin{align}\label{estimate-4*}
& \frac{d}{dt}\sum_{k\in \mathbb{Z}}\int_{\mathbb{R}}\big| \mathcal{M} \hat{u}\big|^2+\gamma^2\big| \mathcal{M} \hat{\vth}\big|^2d\xi+2 C_\g\sum_{k\in \mathbb{Z}\setminus\{0\}}\int_{\mathbb{R}}  |\lan \xi/k-t\ran^{\f{\delta}{2}}\mathcal{M} \hat{u}_2|^2 d\xi\non\\
&+2 C_\g\sum_{k\in \mathbb{Z}\setminus\{0\}}\int_{\mathbb{R}} \g^2 |\lan \xi/k-t\ran^{-1+\f{\delta}{2}}\mathcal{M} \hat{\vth}|^2 d\xi+\Re\sum_{k\in \mathbb{Z}}\int_{\mathbb{R}} \f{ \g^2 k (\xi-kt)}{k^2+(\xi-kt)^2}\big|\mathcal{M} \hat{\vth}\big|^2 d\xi
\non\\
&+\sum_{k\in \mathbb{Z}\setminus\{0\}}\int_{\mathbb{R}}\f{2{\kappa}^{1/3}|k|^{4/3}}{|k|^{2/3}+{\kappa}^{2/3}(\xi-kt)^2}(\big|\mathcal{M} \hat{u}\big|^2+\g^2\big|\mathcal{M} \hat{\vth}\big|^2)d\xi\non\\
&+2 \sum_{k\in \mathbb{Z}}\int_{\mathbb{R}} \Upsilon( |\mathcal{M} \hat{u}|^2+\g^2\big|\mathcal{M} \hat{\vth}\big|^2) d\xi+\sum_{k\in \mathbb{Z}}\int_{\mathbb{R}}(4\kappa^{-\f13} t^{-3}-2\epsilon \kappa^\f13\mathbf{1}_{k\neq0})( \big|\mathcal{M} \hat{u}\big|^2+\g^2\big|\mathcal{M} \hat{\vth}\big|^2)d\xi\non\\
&+2 \sum_{k\in \mathbb{Z}}\int_{\mathbb{R}} (k^2+(\xi-kt)^2)(\nu\big|\mathcal{M} \hat{u}\big|^2+\mu\gamma^2\big|\mathcal{M} \hat{\vth}\big|^2) d\xi+\Re\sum_{k\in \mathbb{Z}}\int_{\mathbb{R}}\mathcal{M} \hat{u}_2\mathcal{M} \overline{\hat{u}}_1 d\xi \nonumber\\
&
=2\Re\sum_{k\in \mathbb{Z}}\int_{\mathbb{R}}\mathcal{M} \hat{F} \mathcal{M} \overline{\hat{u}}d\xi+2\g^2\Re\sum_{k\in \mathbb{Z}}\int_{\mathbb{R}}  \mathcal{M} \hat{G} \mathcal{M} \overline{\hat{\vth}} d\xi.
\end{align}
We calculate the time derivative of  $\mathcal{M} \hat{\vth}M \overline{\hat{u}}_1$ 
\begin{align}\label{c1*}
&\f{d}{dt}\Big( \mathcal{M} \hat{\vth}M \overline{\hat{u}}_1\Big)
= (\mathcal{M} \hat{\vth})_tM \overline{\hat{u}}_1+\mathcal{M} \hat{\vth}(M \overline{\hat{u}}_{1})_t\nonumber\\
=&2q^* \mathcal{M} \hat{\vth}M \overline{\hat{u}}_1-(\nu+\mu) (k^2+(\xi-kt)^2) \mathcal{M} \hat{\vth}M \overline{\hat{u}}_1+\mathcal{M} \hat{G} M \overline{\hat{u}}_1\nonumber\\
&+\f{k^2-(\xi-kt)^2}{k^2+(\xi-kt)^2} \mathcal{M} \hat{\vth}M \overline{\hat{u}}_2+\f{\g^2k(\xi-kt)}{k^2+(\xi-kt)^2} |\mathcal{M} \hat{\vth}|^2
+  \mathcal{M} \hat{u}_2M \overline{\hat{u}}_1+\mathcal{M} \hat{\vth}M\overline{\hat{F}}_1.
\end{align}

 Integrating   with respect to $\xi$, subsequently summing over $k$, taking the real part of \eqref{c1*}, and adding it to  \eqref{estimate-4*}, we have
  \begin{align}\label{rs2*}
& \f{d}{dt}\sum_{k\in \mathbb{Z}}\int_{\mathbb{R}} \Big(\big|\mathcal{M} \hat{u}\big|^2+\g^2\big|\mathcal{M} \hat{\vth}\big|^2+\Re (\mathcal{M} \hat{\vth}\mathcal{M} \overline{\hat{u}}_1)\Big) d\xi+2 C_\g\sum_{k\in \mathbb{Z}\setminus\{0\}}\int_{\mathbb{R}}  |\lan \xi/k-t\ran^{\f{\delta}{2}}\mathcal{M} \hat{u}_2|^2 d\xi\non\\
&+2 C_\g\sum_{k\in \mathbb{Z}\setminus\{0\}}\int_{\mathbb{R}} \g^2 |\lan \xi/k-t\ran^{-1+\f{\delta}{2}}\mathcal{M} \hat{\vth}|^2 d\xi
\non\\
&+\sum_{k\in \mathbb{Z}\setminus\{0\}}\int_{\mathbb{R}}\f{2{\kappa}^{1/3}|k|^{4/3}}{|k|^{2/3}+{\kappa}^{2/3}(\xi-kt)^2}(\big|\mathcal{M} \hat{u}\big|^2+\g^2\big|\mathcal{M} \hat{\vth}\big|^2)d\xi\non\\
&+2 \sum_{k\in \mathbb{Z}}\int_{\mathbb{R}} \Upsilon( |\mathcal{M} \hat{u}|^2+\g^2\big|\mathcal{M} \hat{\vth}\big|^2) d\xi+\sum_{k\in \mathbb{Z}}\int_{\mathbb{R}}(4\kappa^{-\f13} t^{-3}-2\epsilon \kappa^\f13\mathbf{1}_{k\neq0})(  \big|\mathcal{M} \hat{u}\big|^2+\g^2\big|\mathcal{M} \hat{\vth}\big|^2)d\xi\non\\
&+2 \sum_{k\in \mathbb{Z}}\int_{\mathbb{R}} (k^2+(\xi-kt)^2)(\nu\big|\mathcal{M} \hat{u}\big|^2+\mu\gamma^2\big|\mathcal{M} \hat{\vth}\big|^2) d\xi\non\\
& =\Re\sum_{k\in \mathbb{Z}}\int_{\mathbb{R}}\Big(-\underbrace{(\nu+\mu)(k^2+(\xi-kt)^2)(\mathcal{M} \hat{\vth}M \overline{\hat{u}}_1) }_{\mathcal{I}_1^*}\non\\
&+\underbrace{2q^* (\mathcal{M} \hat{\vth}M \overline{\hat{u}}_1) +\f{k^2-(\xi-k t)^2}{k^2+(\xi-kt)^2} (\mathcal{M} \hat{\vth}M \overline{\hat{u}}_2)}_{\mathcal{I}_2^*}\nonumber\\
 &+ \underbrace{\mathcal{M} \hat{G} M \overline{\hat{u}}_1+\mathcal{M} \hat{\vth}M\overline{ \hat{F}}_1}_{\mathcal{I}_3^*}+\underbrace{2\g^2\mathcal{M} \hat{ G} M \overline{\hat{\vth}}}_{\mathcal{I}_4^*} + \underbrace{2\mathcal{M} \hat{F} M \overline{\hat{u}}}_{\mathcal{I}_5^*}\Big) d\xi.
\end{align}
We mention that the estimate of $I_1^*$ is the same as $I_1$, \begin{align}\label{I1*}
{\left|\mathcal{I}_1^*\right|}
\leq&\f{\nu+\mu}{2\g\sqrt{\nu \mu} }(k^2+(\xi-kt)^2)(\nu\big|\mathcal{M}\hat{u}\big|^2+\mu\g^2\big|\mathcal{M}\hat{\vth}\big|^2),
\end{align} and by the divergence free condition $\eqref{ub3}_3$ and $C_\g\geq 1 $, $\g> \f12$,
\begin{align}\label{I2*}
\left|\mathcal{I}_2^*\right|
=&\bigg|2\Big(-\f{ k (\xi-kt)}{2(k^2+(\xi-kt)^2)}+\epsilon \kappa^\f13\mathbf{1}_{k\neq0}-2\kappa^{-\f13} t^{-3}- \f{{\kappa}^{1/3}|k|^{4/3}}{|k|^{2/3}+{\kappa}^{2/3}(\xi-kt)^2}\non\\
&-C_{\g}\lan\xi/k-t\ran^{-2+\delta}- \Upsilon\Big)(\mathcal{M} \hat{\vth}M \overline{\hat{u}}_1)+\f{k^2-(\xi-k t)^2}{k^2+(\xi-kt)^2} (\mathcal{M} \hat{\vth}\mathcal{M} \overline{\hat{u}}_2)\bigg|\nonumber\\
=&\Big|\f{k^2}{k^2+(\xi-kt)^2} (\mathcal{M} \hat{\vth}M \overline{\hat{u}}_2)+2\big(\epsilon \kappa^\f13\mathbf{1}_{k\neq0}-2\kappa^{-\f13} t^{-3}- \f{{\kappa}^{1/3}|k|^{4/3}}{|k|^{2/3}+{\kappa}^{2/3}(\xi-kt)^2}\non\\
&-C_{\g}\lan\xi/k-t\ran^{-2+\delta}- \Upsilon\big)(\mathcal{M} \hat{\vth}M \overline{\hat{u}}_1)\Big|\nonumber\\
\leq& \f{1}{\g}\f{k^2}{k^2+(\xi-kt)^2}(\g^2|\mathcal{M} \hat{\vth}|^2+|\mathcal{M} \hat{u}_2|^2)\non\\
&+\f{1}{\g}\Big(2\kappa^{-\f13} t^{-3}+\epsilon \kappa^\f13\mathbf{1}_{k\neq0}+\f{{\kappa}^{1/3}|k|^{4/3}}{|k|^{2/3}+{\kappa}^{2/3}(\xi-kt)^2}+ \Upsilon\Big)(\g^2|\mathcal{M} \hat{\vth}|^2+|\mathcal{M} \hat{u}_1|^2)\non\\
&+\f{C_\g}{\g}\lan\xi/k-t\ran^{\delta}\f{k^2}{k^2+(\xi-kt)^2}(\g^2|\mathcal{M} \hat{\vth}|^2+\f{|\xi-kt|^2}{k^2}|\mathcal{M} \hat{u}_2|^2)\mathbf{1}_{k\neq0}\nonumber\\
\leq &\f{C_{\g}+1}{\g}\lan\xi/k-t\ran^{-2+\delta} \g^2|\mathcal{M} \hat{\vth}|^2+\f{C_{\g}}{\g}\lan\xi/k-t\ran^{\delta}|\mathcal{M} \hat{u}_2|^2\\
&\non+\f{1}{\g}\Big(2\kappa^{-\f13} t^{-3}+{\epsilon \kappa^\f13\mathbf{1}_{k\neq0}}+\f{{\kappa}^{1/3}|k|^{4/3}}{|k|^{2/3}+{\kappa}^{2/3}(\xi-kt)^2}+ \Upsilon\Big)(\g^2|\mathcal{M} \hat{\vth}|^2+|\mathcal{M} \hat{u}|^2).
\end{align}
 So, taking \eqref{I1*}, \eqref{I2*} into  \eqref{rs2*},  $\f{1}{\g}\leq \f{\nu+\mu}{2\g\sqrt{\nu \mu} }< 2-\e<2$ and $\f{C_{\g}+2}{\g}\leq 2C_\g$ (as $\f{2}{2\g-1}\leq C_\g$), we have
\begin{align}\label{rs4*}
&\f{d}{dt}\sum_{k\in \mathbb{Z}}\int_{\mathbb{R}} \Big(\big|\mathcal{M}\hat{u}\big|^2+\g^2\big|\mathcal{M}\hat{\vth}\big|^2+\Re (\mathcal{M}\hat{\vth}\mathcal{M} \overline{\hat{u}}_1)\Big) d\xi 
+\f{2}{\g}\sum_{k\in \mathbb{Z}\setminus\{0\}}\int_{\mathbb{R}}  |\lan \xi/k-t\ran^{\f{\delta}{2}}\mathcal{M} \hat{u}_2|^2 d\xi\nonumber\\
&+\e\sum_{k\in \mathbb{Z}}\int_{\mathbb{R}}\f{{\kappa}^{1/3}|k|^{4/3}}{|k|^{2/3}+{\kappa}^{2/3}(\xi-kt)^2}(\big|\mathcal{M} \hat{u}\big|^2+\g^2\big|\mathcal{M} \hat{\vth}\big|^2)d\xi\nonumber\\
&-{4}\epsilon \kappa^\f13\sum_{k\in \mathbb{Z}\setminus\{0\}}\int_{\mathbb{R}}(  \big|\mathcal{M} \hat{u}\big|^2+\g^2\big|\mathcal{M} \hat{\vth}\big|^2)d\xi
\nonumber\\
&+\e\sum_{k\in \mathbb{Z}}\int_{\mathbb{R}}\Upsilon(  |\mathcal{M} \hat{u}|^2+\g^2|\mathcal{M} \hat{\vth}|^2) d\xi+
\e\kappa^{-\f13}t^{-3}\sum_{k\in \mathbb{Z}}\int_{\mathbb{R}}(  |\mathcal{M} \hat{u}|^2+\g^2|\mathcal{M} \hat{\vth}|^2) d\xi\nonumber\\
&+\e \sum_{k\in \mathbb{Z}}\int_{\mathbb{R}} (k^2+(\xi-kt)^2)(\nu\big|\mathcal{M}\hat{u}\big|^2+\mu\gamma^2\big|\mathcal{M}\hat{\vth}\big|^2) d\xi \nonumber\\
 \leq&\Re\sum_{k\in \mathbb{Z}}\int_{\mathbb{R}}  \Big(\mathcal{I}_3^*+\mathcal{I}_4^*+\mathcal{I}_5^*\Big) d\xi.
\end{align}
 Now we take $\epsilon=\e/16 $ 
 and use $\kappa=\min\{\nu,\mu\} $,
 \begin{align}\non
\f{{\kappa}^{1/3}|k|^{4/3}}{|k|^{2/3}+{\kappa}^{2/3}(\xi-kt)^2}+\kappa(\xi-kt)^2&\geq
\f{{\kappa}^{1/3}|k|^{4/3}}{|k|^{2/3}+{\kappa}^{2/3}(\xi-kt)^2}+\f{|k|^{2/3}\kappa(\xi-kt)^2}{|k|^{2/3}+{\kappa}^{2/3}(\xi-kt)^2}\\&=
|k|^{\f23}\kappa^{\f13}\label{nu1}
\end{align}
to deduce that\begin{align}\label{rs4a}
&\f{d}{dt}\sum_{k\in \mathbb{Z}}\int_{\mathbb{R}} \Big(\big|\mathcal{M}\hat{u}\big|^2+\g^2\big|\mathcal{M}\hat{\vth}\big|^2+\Re (\mathcal{M}\hat{\vth}\mathcal{M} \overline{\hat{u}}_1)\Big) d\xi 
+\f{2}{\g}\sum_{k\in \mathbb{Z}\setminus\{0\}}\int_{\mathbb{R}}  |\lan \xi/k-t\ran^{\f{\delta}{2}}\mathcal{M} \hat{u}_2|^2 d\xi\nonumber\\
&+\f{\e}{4} \sum_{k\in \mathbb{Z}}\int_{\mathbb{R}}\kappa^{\f13}|k|^{\f23}(  \big|\mathcal{M} \hat{u}\big|^2+\g^2\big|\mathcal{M} \hat{\vth}\big|^2)d\xi
\nonumber\\
&+\e\sum_{k\in \mathbb{Z}}\int_{\mathbb{R}}\Upsilon(  |\mathcal{M} \hat{u}|^2+\g^2|\mathcal{M} \hat{\vth}|^2) d\xi+
\e\kappa^{-\f13}t^{-3}\sum_{k\in \mathbb{Z}}\int_{\mathbb{R}}(  |\mathcal{M} \hat{u}|^2+\g^2|\mathcal{M} \hat{\vth}|^2) d\xi\nonumber\\
&+\f{\e}{2} \sum_{k\in \mathbb{Z}}\int_{\mathbb{R}} (k^2+(\xi-kt)^2){\kappa}(\big|\mathcal{M}\hat{u}\big|^2+\gamma^2\big|\mathcal{M}\hat{\vth}\big|^2) d\xi \nonumber\\
 \leq&\Re\sum_{k\in \mathbb{Z}}\int_{\mathbb{R}}  \Big(\mathcal{I}_3^*+\mathcal{I}_4^*+\mathcal{I}_6^*\Big) d\xi.
\end{align}
Here $\mathcal{I}_3^*:={\mathcal{M} \hat{G} \mathcal{M} \overline{\hat{u}}_1}+\mathcal{M}\hat{\vth}\mathcal{M}\overline{ \hat{F}}_1,\quad \mathcal{I}_4^*:={2\g^2\mathcal{M} \hat{ G}  \mathcal{M} \overline{\hat{\vth}}},\quad  \mathcal{I}_5^*:={2\mathcal{M} \hat{F} \mathcal{M} \overline{\hat{u}}}.$

For the nonlinear term of \eqref{rs4*}, we use  Lemma \ref{lem2} and Lemma \ref{lem7} to obtain
\begin{align*}
&\Big|\Re\sum_{k\in \mathbb{Z}}\int_{\mathbb{R}}\mathcal{I}_4^* d\xi\Big|=\Big|2\g^2\Re\sum_{k\in \mathbb{Z}}\int_{\mathbb{R}}(\mathcal{M} \hat{G} \mathcal{M} \overline{\hat{\vth}})d\xi\Big|\\
=&\Big|2\g^2\Re\sum_{k, l\in \mathbb{Z}}\int_{\mathbb{R}^2}( \hat{u}_1(k-l, \xi-\eta )il\hat{\vth}(l, \eta) +  \hat{u}_2(k-l, \xi-\eta) i (\eta-l t)\hat{\vth}(l, \eta) )\mathcal{M}^2\overline{\hat{\vth}}(k, \xi)d\xi d\eta\Big| \\
\leq &\sum_{k, l\in \mathbb{Z}}\int_{\mathbb{R}^2}\g^2\Big[\Big|l \mathcal{M}^2 (t,k, \xi)-k\mathcal{M}^2 (t, l,\eta) \Big||\hat{u}_1(k-l, \xi-\eta )||\hat{\vth}(l, \eta)||\overline{\hat{\vth}}(k, \xi) | \\
&+2\Big|(\eta-lt) \mathcal{M}^2 (t,k, \xi)\Big| |\hat{u}_2(k-l, \xi-\eta )||\hat{\vth}(l, \eta)||\overline{\hat{\vth}}(k, \xi)|\Big]d\xi d\eta \\
\leq& C\sum_{k, l\in \mathbb{Z}}\int_{\mathbb{R}^2}\mathcal{M}(t, k,\xi)\mathcal{M}(t, l,\eta)\mathcal{M}(t, k-l,\xi-\eta) (\lan k, \xi\ran^{\f12-s}+\lan l, \eta\ran^{\f12-s}+\lan k-l, \xi-\eta\ran^{\f12-s})\times\\&(|k|^{\f13}|k-l|^{\f13}+|l|^{\f13}|k-l|^{\f13}+|kl|^{\f13}+\kappa^{\f23}|kl|)|\hat{u}(k-l, \xi-\eta )||\hat{\vth}(l, \eta)||\overline{\hat{\vth}}(k, \xi)| d\xi d\eta\\&+C\|e^{\epsilon \kappa^\f13 t}\hat{u}_2\|_{L^1_{k, \xi}}\||k,\xi-kt|\mathcal{M}{\hat{\vth}}\|_{L^2_{k, \xi}}\|\mathcal{M}\hat{\vth}\|_{L^2_{k, \xi}}\\&+C\|\mathcal{M}\hat{u}_2\|_{L^2_{k, \xi}}(\kappa^{\f16}\||k,\xi-kt|\mathcal{M}{\hat{\vth}}\|_{L^2_{k, \xi}}+\kappa^{-\f16}\||k|^{\f13}\mathcal{M}{\hat{\vth}}\|_{L^2_{k, \xi}})\|\mathcal{M}\hat{\vth}\|_{L^2_{k, \xi}}\\
 &+C\kappa^{-\f13}\|\lan \xi/k-t\ran^{\f{\delta}{2}}\mathcal{M}\hat{u}_2\|_{L^2_{k, \xi}}\|\mathcal{M}\hat{\vth}\|_{L_{k,\xi}^2}\|\sqrt{\Upsilon}\mathcal{M}\hat{\vth}\|_{L_{k,\xi}^2}
\\ \leq &C\||k|^{\f13}\mathcal{M}\hat{u}\|_{L^2_{k, \xi}}
\||k|^{\f13}\mathcal{M}{\hat{\vth}}\|_{L^2_{k, \xi}}\|\mathcal{M}\hat{\vth}\|_{L^2_{k, \xi}}+C\|\mathcal{M}\hat{u}\|_{L^2_{k, \xi}}
(\||k|^{\f13}\mathcal{M}{\hat{\vth}}\|_{L^2_{k, \xi}}^2+\kappa^{\f23}\||k|\mathcal{M}{\hat{\vth}}\|_{L^2_{k, \xi}}^2)\\
 &+C t^{-\f32}\|\mathcal{M}\hat{u}\|_{L^2_{k, \xi}}\||k,\xi-kt|\mathcal{M}{\hat{\vth}}\|_{L^2_{k, \xi}}\|\mathcal{M}\hat{\vth}\|_{L^2_{k, \xi}}\\&+C(\kappa^{-\f13}\|\mathcal{M}\hat{u}_2\|_{L^2_{k, \xi}}^2+\kappa^{\f23}\||k,\xi-kt|\mathcal{M}{\hat{\vth}}\|_{L^2_{k, \xi}}^2+\||k|^{\f13}\mathcal{M}{\hat{\vth}}\|_{L^2_{k, \xi}}^2)\|\mathcal{M}\hat{\vth}\|_{L^2_{k, \xi}}\\
 &+C\kappa^{-\f13}(\|\lan \xi/k-t\ran^{\f{\delta}{2}}\mathcal{M}\hat{u}_2\|_{L^2_{k, \xi}}^2+\|\sqrt{\Upsilon}\mathcal{M}\hat{\vth}\|_{L_{k,\xi}^2}^2)\|\mathcal{M}\hat{\vth}\|_{L_{k,\xi}^2}\\
 \leq &C\|\mathcal{M}(\hat{u},\hat{\vth})\|_{L^2_{k, \xi}}(\||k|^{\f13}\mathcal{M}(\hat{u},\hat{\vth})\|_{L^2_{k, \xi}}^2+\kappa^{\f23}\||k,\xi-kt|\mathcal{M}\hat{\vth}\|_{L^2_{k, \xi}}^2+\kappa^{-\f23}t^{-3}\|\mathcal{M}\hat{u}\|_{L^2_{k, \xi}}^2\\&+
 \kappa^{-\f13}\|\lan \xi/k-t\ran^{\f{\delta}{2}}\mathcal{M}\hat{u}_2\|_{L^2_{k, \xi}}^2+\kappa^{-\f13}\|\sqrt{\Upsilon}\mathcal{M}\hat{\vth}\|_{L_{k,\xi}^2}^2),
\end{align*}
where we used (recall that $\mathcal{M}(t, k, \xi)\sim \lan k, \xi-kt\ran^\frac{1}{2} 
\lan k, \xi\ran^se^{\epsilon\kappa^\f13 t\mathbf{1}_{k\neq 0}}$)
\beno
\|e^{\epsilon \kappa^\f13 t}\hat{ u}_2\|_{L^1}\leq C\lan t\ran^{-\f32}\||k_+, \xi-kt|^\frac{1}{2} \lan k, \xi\ran^se^{\epsilon \kappa^\f13 t} \hat{ u}\mathbf{1}_{k\neq0}\|_{L^2}\leq C\lan t\ran^{-\f32} \|\mathcal{M} \hat{u}\|_{L^2},
\eeno
by Lemma  \ref{Lem3}.

For $I^*_3$, we use symmetrization, Lemma \ref{lem2} and Lemma \ref{lem7} to obtain
\begin{align*}
&\Big|\Re\sum_{k\in \mathbb{Z}}\int_{\mathbb{R}}\mathcal{I}_3^*d\xi\Big|=\Big|\Re\sum_{k\in \mathbb{Z}}\int_{\mathbb{R}}(\mathcal{M} \hat{G} \mathcal{M} \overline{\hat{u}}_1+\mathcal{M}\hat{\vth}\mathcal{M}\overline{ \hat{F}}_1)d\xi\Big|\\
=&\Big|\Re\sum_{k, l\in \mathbb{Z}}\int_{\mathbb{R}^2}\Big[( \hat{u}_1(k-l, \xi-\eta )il\hat{\vth}(l, \eta) +  \hat{u}_2(k-l, \xi-\eta) i (\eta-l t)\hat{\vth}(l, \eta) )\mathcal{M}^2\overline{\hat{u}}_1(k, \xi) \\
&-\mathcal{M}^2\hat{\vth}(l, \eta)\big( \hat{u}_1(k-l, \xi-\eta )ik\overline{\hat{u}}_1(k, \xi) +  \hat{u}_2(k-l, \xi-\eta) i (\xi-kt)\overline{\hat{u}}_1(k, \xi) \nonumber\\
&-\f{2il}{l^2+(\eta-lt)^2}
\hat{u}_2(k-l, \xi-\eta )(\xi-kt)(l\overline{\hat{u}}_1(k, \xi)+(\eta-lt)\overline{\hat{u}}_2(k, \xi)) \big) \Big]d\xi d\eta\Big| \nonumber\\
\leq &\sum_{k, l\in \mathbb{Z}}\int_{\mathbb{R}^2}\Big[\Big|l \mathcal{M}^2 (t, k, \xi)-k\mathcal{M} ^2 (t, l,  \eta) \Big||\hat{u}_1(k-l, \xi-\eta )||\hat{\vth}(l, \eta)||\overline{\hat{u}}_1(k, \xi) | \nonumber\\
&+(|(\eta-lt) \mathcal{M}^2 (t, k, \xi)|+|(\xi-kt) \mathcal{M}^2 (t, l, \eta)|) |\hat{u}_2(k-l, \xi-\eta )||\hat{\vth}(l, \eta)||\overline{\hat{u}}_1(k, \xi)|\Big]d\xi d\eta\\
\leq& C\sum_{k, l\in \mathbb{Z}}\int_{\mathbb{R}^2}\mathcal{M}(t, k,\xi)\mathcal{M}(t, l,\eta)\mathcal{M}(t, k-l,\xi-\eta) (\lan k, \xi\ran^{\f12-s}+\lan l, \eta\ran^{\f12-s}+\lan k-l, \xi-\eta\ran^{\f12-s})\times\\&(|k|^{\f13}|k-l|^{\f13}+|l|^{\f13}|k-l|^{\f13}+|kl|^{\f13}+\kappa^{\f23}|kl|)|\hat{u}(k-l, \xi-\eta )||\hat{\vth}(l, \eta)||\overline{\hat{u}}_1(k, \xi)| d\xi d\eta\\&+C\|e^{\epsilon \kappa^\f13 t}\hat{u}_2\|_{L^1_{k, \xi}}(\||k,\xi-kt|\mathcal{M}{\hat{\vth}}\|_{L^2_{k, \xi}}\|\mathcal{M}\hat{u}\|_{L^2_{k, \xi}}+\|\mathcal{M}{\hat{\vth}}\|_{L^2_{k, \xi}}\||k,\xi-kt|\mathcal{M}\hat{u}\|_{L^2_{k, \xi}})\\&+C\|\mathcal{M}\hat{u}_2\|_{L^2_{k, \xi}}(\kappa^{\f16}\||k,\xi-kt|\mathcal{M}{\hat{\vth}}\|_{L^2_{k, \xi}}+\kappa^{-\f16}\||k|^{\f13}\mathcal{M}{\hat{\vth}}\|_{L^2_{k, \xi}})\|\mathcal{M}\hat{u}\|_{L^2_{k, \xi}}\\&+C\|\mathcal{M}\hat{u}_2\|_{L^2_{k, \xi}}(\kappa^{\f16}\||k,\xi-kt|\mathcal{M}\hat{u}\|_{L^2_{k, \xi}}+\kappa^{-\f16}\||k|^{\f13}\mathcal{M}\hat{u}\|_{L^2_{k, \xi}})\|\mathcal{M}{\hat{\vth}}\|_{L^2_{k, \xi}}\\
 &+C\kappa^{-\f13}\|\lan \xi/k-t\ran^{\f{\delta}{2}}\mathcal{M}\hat{u}_2\|_{L^2_{k, \xi}}(\|\mathcal{M}\hat{\vth}\|_{L_{k,\xi}^2}\|\sqrt{\Upsilon}\mathcal{M}\hat{u}\|_{L_{k,\xi}^2}+
 \|\sqrt{\Upsilon}\mathcal{M}\hat{\vth}\|_{L_{k,\xi}^2}\|\mathcal{M}\hat{u}\|_{L_{k,\xi}^2})\\
\leq &C\||k|^{\f13}\mathcal{M}\hat{u}\|_{L^2_{k, \xi}}
\||k|^{\f13}\mathcal{M}{\hat{\vth}}\|_{L^2_{k, \xi}}\|\mathcal{M}\hat{u}\|_{L^2_{k, \xi}}+C\|\mathcal{M}\hat{\vth}\|_{L^2_{k, \xi}}
\||k|^{\f13}\mathcal{M}{\hat{u}}\|_{L^2_{k, \xi}}^2\\&+C\kappa^{\f23}\|\mathcal{M}\hat{u}\|_{L^2_{k, \xi}}\||k|\mathcal{M}{\hat{u}}\|_{L^2_{k, \xi}}\||k|\mathcal{M}{\hat{\vth}}\|_{L^2_{k, \xi}}\\
 &+Ct^{-\f32}\|\mathcal{M}\hat{u}\|_{L^2_{k, \xi}}(\||k,\xi-kt|\mathcal{M}{\hat{\vth}}\|_{L^2_{k, \xi}}\|\mathcal{M}\hat{u}\|_{L^2_{k, \xi}}+\|\mathcal{M}{\hat{\vth}}\|_{L^2_{k, \xi}}\||k,\xi-kt|\mathcal{M}\hat{u}\|_{L^2_{k, \xi}})\\&
 +C(\kappa^{-\f13}\|\mathcal{M}\hat{u}_2\|_{L^2_{k, \xi}}^2+\kappa^{\f23}\||k,\xi-kt|\mathcal{M}{\hat{\vth}}\|_{L^2_{k, \xi}}^2+\||k|^{\f13}\mathcal{M}{\hat{\vth}}\|_{L^2_{k, \xi}}^2)\|\mathcal{M}\hat{u}\|_{L^2_{k, \xi}}\\&+C(\kappa^{-\f13}\|\mathcal{M}\hat{u}_2\|_{L^2_{k, \xi}}^2+\kappa^{\f23}\||k,\xi-kt|\mathcal{M}\hat{u}\|_{L^2_{k, \xi}}^2+\||k|^{\f13}\mathcal{M}\hat{u}\|_{L^2_{k, \xi}}^2)\|\mathcal{M}{\hat{\vth}}\|_{L^2_{k, \xi}}\\
 &+C\kappa^{-\f13}\|\lan \xi/k-t\ran^{\f{\delta}{2}}\mathcal{M}\hat{u}_2\|_{L^2_{k, \xi}}(\|\mathcal{M}\hat{\vth}\|_{L_{k,\xi}^2}\|\sqrt{\Upsilon}\mathcal{M}\hat{u}\|_{L_{k,\xi}^2}+
 \|\sqrt{\Upsilon}\mathcal{M}\hat{\vth}\|_{L_{k,\xi}^2}\|\mathcal{M}\hat{u}\|_{L_{k,\xi}^2})\\
 \leq &C\|\mathcal{M}(\hat{u},\hat{\vth})\|_{L^2_{k, \xi}}(\||k|^{\f13}\mathcal{M}(\hat{u},\hat{\vth})\|_{L^2_{k, \xi}}^2+\kappa^{\f23}\||k,\xi-kt|\mathcal{M}(\hat{u},\hat{\vth})\|_{L^2_{k, \xi}}^2+\kappa^{-\f23}t^{-3}\|\mathcal{M}\hat{u}\|_{L^2_{k, \xi}}^2\\&+
 \kappa^{-\f13}\|\lan \xi/k-t\ran^{\f{\delta}{2}}\mathcal{M}\hat{u}_2\|_{L^2_{k, \xi}}^2+\kappa^{-\f13}\|\sqrt{\Upsilon}\mathcal{M}(\hat{u},\hat{\vth})\|_{L_{k,\xi}^2}^2),
\end{align*}
here we used  that
$
-\Delta_Lp_{NL}=2\div_L(u_2\na_L^2u).
$

For $I^*_5$, we use divergence free condition $ \eqref{ub3}_3$, Lemma \ref{lem2} and Lemma \ref{lem7} to obtain
\begin{align*}
&\Big|\Re\sum_{k\in \mathbb{Z}}\int_{\mathbb{R}}\mathcal{I}_5^* d\xi\Big|=\Big|2\Re\sum_{k\in \mathbb{Z}}\int_{\mathbb{R}}\mathcal{M} \hat{F} \mathcal{M} \overline{\hat{u}}d\xi\Big|\nonumber\\
=&\Big|2\Re\sum_{k, l\in \mathbb{Z}}\int_{\mathbb{R}^2}\Big[( \hat{u}_1(k-l, \xi-\eta )il\hat{u}(l, \eta) +  \hat{u}_2(k-l, \xi-\eta) i (\eta-l t)\hat{u}(l, \eta) \Big] \mathcal{M}^2\overline{\hat{u}}(k, \xi)d\xi d\eta
\non\\
&+2\Re\sum_{k\in \mathbb{Z}}\int_{\mathbb{R}}(ik\hat{p}_{NL} \mathcal{M}^2\overline{\hat{u}}_1(k, \xi)+
i(\xi-kt)\hat{p}_{NL} \mathcal{M}^2\overline{\hat{u}}_2(k, \xi))
d\xi\Big| \nonumber\\
\leq &\sum_{k, l\in \mathbb{Z}}\int_{\mathbb{R}^2}\Big[\Big|l \mathcal{M}^2 (t,k, \xi)-k\mathcal{M}^2 (t, l,\eta) \Big||\hat{u}_1(k-l, \xi-\eta )||\hat{u}(l, \eta)||\overline{\hat{u}}(k, \xi) | \\
&+2\Big|(\eta-lt) \mathcal{M}^2 (t,k, \xi)\Big| |\hat{u}_2(k-l, \xi-\eta )||\hat{u}(l, \eta)||\overline{\hat{u}}(k, \xi)|\Big]d\xi d\eta \\
\leq& C\sum_{k, l\in \mathbb{Z}}\int_{\mathbb{R}^2}\mathcal{M}(t, k,\xi)\mathcal{M}(t, l,\eta)\mathcal{M}(t, k-l,\xi-\eta) (\lan k, \xi\ran^{\f12-s}+\lan l, \eta\ran^{\f12-s}+\lan k-l, \xi-\eta\ran^{\f12-s})\times\\&(|k|^{\f13}|k-l|^{\f13}+|l|^{\f13}|k-l|^{\f13}+|kl|^{\f13}+\kappa^{\f23}|kl|)|\hat{u}(k-l, \xi-\eta )||\hat{u}(l, \eta)||\overline{\hat{u}}(k, \xi)| d\xi d\eta\\&+C\|e^{\epsilon \kappa^\f13 t}\hat{u}_2\|_{L^1_{k, \xi}}\||k,\xi-kt|\mathcal{M}{\hat{u}}\|_{L^2_{k, \xi}}\|\mathcal{M}\hat{u}\|_{L^2_{k, \xi}}\\&+C\|\mathcal{M}\hat{u}_2\|_{L^2_{k, \xi}}(\kappa^{\f16}\||k,\xi-kt|\mathcal{M}{\hat{u}}\|_{L^2_{k, \xi}}+\kappa^{-\f16}\||k|^{\f13}\mathcal{M}{\hat{u}}\|_{L^2_{k, \xi}})\|\mathcal{M}\hat{u}\|_{L^2_{k, \xi}}\\
 &+C\kappa^{-\f13}\|\lan \xi/k-t\ran^{\f{\delta}{2}}\mathcal{M}\hat{u}_2\|_{L^2_{k, \xi}}\|\mathcal{M}\hat{u}\|_{L_{k, \xi}^2}\|\sqrt{\Upsilon}\mathcal{M}\hat{u}\|_{L_{k,\xi}^2}\\
\leq &
C\|\mathcal{M}\hat{u}\|_{L^2_{k, \xi}}
(\||k|^{\f13}\mathcal{M}{\hat{u}}\|_{L^2_{k, \xi}}^2+\kappa^{\f23}\||k|\mathcal{M}{\hat{u}}\|_{L^2_{k, \xi}}^2)\\&+Ct^{-\f32}\|\mathcal{M}\hat{u}\|_{L^2_{k, \xi}}\||k,\xi-kt|\mathcal{M}{\hat{u}}\|_{L^2_{k, \xi}}\|\mathcal{M}\hat{u}\|_{L^2_{k, \xi}}\\&+C(\kappa^{-\f13}\|\mathcal{M}\hat{u}_2\|_{L^2_{k, \xi}}^2+\kappa^{\f23}\||k,\xi-kt|\mathcal{M}{\hat{u}}\|_{L^2_{k, \xi}}^2+\||k|^{\f13}\mathcal{M}{\hat{u}}\|_{L^2_{k, \xi}}^2)\|\mathcal{M}\hat{u}\|_{L^2_{k, \xi}}\\
 &+C\kappa^{-\f13}(\|\lan \xi/k-t\ran^{\f{\delta}{2}}\mathcal{M}\hat{u}_2\|_{L^2_{k, \xi}}^2+\|\sqrt{\Upsilon}\mathcal{M}\hat{u}\|_{L_{k,\xi}^2}^2)\|\mathcal{M}\hat{u}\|_{L_{k, \xi}^2}\\
 \leq &C\|\mathcal{M}\hat{u}\|_{L^2_{k, \xi}}(\||k|^{\f13}\mathcal{M}\hat{u}\|_{L^2_{k, \xi}}^2+\kappa^{\f23}\||k,\xi-kt|\mathcal{M}\hat{u}\|_{L^2_{k, \xi}}^2+\kappa^{-\f23}t^{-3}\|\mathcal{M}\hat{u}\|_{L^2_{k, \xi}}^2\\&+
 \kappa^{-\f13}\|\lan \xi/k-t\ran^{\f{\delta}{2}}\mathcal{M}\hat{u}_2\|_{L^2_{k, \xi}}^2+\kappa^{-\f13}\|\sqrt{\Upsilon}\mathcal{M}\hat{u}\|_{L_{k,\xi}^2}^2).
\end{align*}Summing up and using \eqref{rs4a}, we obtain Proposition \ref{prop: long-time i}.
 \end{proof}

\begin{proof}[Proof of Proposition \ref{pr02}] 
Assume that $\ t\geq T_0 =\kappa^{-\f16},$ and
$\|\mathcal{M} (\hat{u}, \hat{\vth})(t)\|_{L_{k,\xi}^2} \leq c_2\kappa^\f13,$ 
then by Proposition \ref{prop: long-time i}, we have
\begin{align*}
        &\f{d}{dt}\sum_{k\in \mathbb{Z}}\int_{\mathbb{R}} \Big(\big|\mathcal{M}\hat{u}\big|^2+\g^2\big|\mathcal{M}\hat{\vth}\big|^2+\Re (\mathcal{M}\hat{\vth}\mathcal{M} \overline{\hat{u}}_1)\Big) d\xi 
+\f{2}{\g}\sum_{k\in \mathbb{Z}\setminus\{0\}}\int_{\mathbb{R}}  |\lan \xi/k-t\ran^{\f{\delta}{2}}\mathcal{M} \hat{u}_2|^2 d\xi\nonumber\\
&+\f{\e}{4} \sum_{k\in \mathbb{Z}}\int_{\mathbb{R}}\kappa^{\f13}|k|^{\f23}(  \big|\mathcal{M} \hat{u}\big|^2+\g^2\big|\mathcal{M} \hat{\vth}\big|^2)d\xi
\nonumber\\
&+\e\sum_{k\in \mathbb{Z}}\int_{\mathbb{R}}\Upsilon(  |\mathcal{M} \hat{u}|^2+\g^2|\mathcal{M} \hat{\vth}|^2) d\xi+
\e\kappa^{-\f13}t^{-3}\sum_{k\in \mathbb{Z}}\int_{\mathbb{R}}(  |\mathcal{M} \hat{u}|^2+\g^2|\mathcal{M} \hat{\vth}|^2) d\xi\nonumber\\
&+\f{\e}{2} \sum_{k\in \mathbb{Z}}\int_{\mathbb{R}} (k^2+(\xi-kt)^2)\kappa(\big|\mathcal{M}\hat{u}\big|^2+\gamma^2\big|\mathcal{M}\hat{\vth}\big|^2) d\xi \nonumber\\
         \leq &C^* c_2\kappa^\f13 \bigg(\||k|^{\f13}\mathcal{M}(\hat{u},\hat{\vth})\|_{L^2_{k, \xi}}^2+\kappa^{\f23}\||k,\xi-kt|\mathcal{M}(\hat{u},\hat{\vth})\|_{L^2_{k, \xi}}^2+\kappa^{-\f23}t^{-3}\|\mathcal{M}\hat{u}\|_{L^2_{k, \xi}}^2\non\\&+
 \kappa^{-\f13}\|\lan \xi/k-t\ran^{\f{\delta}{2}}\mathcal{M}\hat{u}_2\|_{L^2_{k, \xi}}^2+\kappa^{-\f13}\|\sqrt{\Upsilon}\mathcal{M}(\hat{u},\hat{\vth})\|_{L_{k,\xi}^2}^2 \bigg),
\end{align*}
 where $C^*>1$ depends only on $\g,s,\e$. Now we take $c_2= {\e_1}/{C^*}$,  $\e_1=\e/4$. Then we obtain
  $$\f{d}{dt}\sum_{k\in \mathbb{Z}}\int_{\mathbb{R}} \Big(\big|\mathcal{M}\hat{u}\big|^2+\g^2\big|\mathcal{M}\hat{\vth}\big|^2+\Re (\mathcal{M}\hat{\vth}\mathcal{M} \overline{\hat{u}}_1)\Big) d\xi \leq 0.$$
 Here we used $\e_1=\e/4<\e/2<\e<1/\g<2/\g $. Thus, we obtain Proposition \ref{pr02}.
          \end{proof}


\end{document}